\documentclass{amsart}
\usepackage[numbers,sort&compress]{natbib}
\usepackage{amsmath,amssymb,amsthm,amscd,mathrsfs}
\usepackage{soul} 

\DeclareMathOperator{\Hom}{Hom}
\DeclareMathOperator{\Ker}{Ker}

\DeclareMathOperator{\Trace}{Tr}

\newcommand{\id}        {\mathsf{id}}
\newcommand{\op}        {\mathrm{op}}
\newcommand{\Vect}      {\mathsf{Vec}}
\newcommand{\fdim}      {\mathit{fd}}
\newcommand{\Trans}     {\mathsf{T}}
\newcommand{\Bil}       {\mathrm{Bil}}

\newcommand{\Rep}       {\mathrm{Rep}}
\newcommand{\Mod}       {\mathsf{mod}}
\newcommand{\fdMod}     {\mathsf{mod}_{\mathit{fd}}}
\newcommand{\Com}       {\mathsf{com}}
\newcommand{\fdCom}     {\mathsf{com}_{\mathit{fd}}}
\newcommand{\End}       {\mathrm{End}}
\newcommand{\coEnd}     {\mathrm{End}^{\mathrm{c}}}
\newcommand{\reg}       {\mathsf{R}}

\numberwithin{equation}{section}

\newtheorem{counter}            {}[section]
\theoremstyle{definition}
\newtheorem{definition}         [counter]{Definition}
\theoremstyle{plain}
\newtheorem{lemma}              [counter]{Lemma}
\newtheorem{proposition}        [counter]{Proposition}
\newtheorem{theorem}            [counter]{Theorem}
\newtheorem{corollary}          [counter]{Corollary}
\newtheorem{question}           [counter]{Question}

\theoremstyle{remark}
\newtheorem{remark}             [counter]{Remark}
\newtheorem{example}            [counter]{Example}

\theoremstyle{plain}
\newtheorem*{theorem*}{Theorem}
\theoremstyle{remark}
\title{Frobenius--Schur Indicator for Categories with Duality}

\author{Kenichi Shimizu}

\address{Graduate School of Mathematics, Nagoya University, Furo-cho, Chikusa-ku, Nagoya 464-8602, Japan}

\email{x12005i@math.nagoya-u.ac.jp}

\keywords{Frobenius--Schur indicator; category with duality; Hopf algebra; quantum groups}

\begin{document}

\begin{abstract}
  We introduce the Frobenius--Schur indicator for categories with duality to give a category-theoretical understanding of various generalizations of the Frobenius--Schur theorem including that for semisimple quasi-Hopf algebras, weak Hopf $C^*$-algebras and association schemes. Our framework also clarifies a mechanism of how the ``twisted'' theory arises from the ordinary case. As a demonstration, we establish twisted versions of the Frobenius--Schur theorem for various algebraic objects. We also give several applications to the quantum $SL_2$.
\end{abstract}

\maketitle

\section{Introduction}
\label{sec:introduction}

The aim of this paper is to develop a category-theoretical framework to unify various generalizations of the {\em Frobenius--Schur theorem}. We first recall the Frobenius--Schur theorem for compact groups. Let $G$ be a compact group, and let $V$ be a finite-dimensional continuous representation of $G$ with character $\chi_V$. The {\em $n$-th Frobenius--Schur indicator} (or {\em FS indicator}, for short) of $V$ is defined and denoted by
\begin{equation}
  \label{eq:FS-formula-cpt}
  \nu_n(V) = \int_G \chi_V(g^n) d \mu(g),
\end{equation}
where $\mu$ is the normalized Haar measure on $G$. The Frobenius--Schur theorem states that the value of the second FS indicator $\nu_2(V)$ has the following meaning:

\begin{theorem*}[Frobenius--Schur theorem]
  If $V$ is irreducible, then we have
  \begin{equation}
    \label{eq:FS-ind-RCH}
    \nu_2(V) = \begin{cases}
      +1 & \text{if $V$ is real}, \\
      \phantom{+} 0 & \text{if $V$ is complex}, \\
      -1 & \text{if $V$ is quaternionic}.
    \end{cases}
  \end{equation}
  Moreover, the following statements are equivalent: \\
  \indent {\rm (1)} $\nu_2(V) \ne 0$. \\
  \indent {\rm (2)} $V$ is isomorphic to its dual representation. \\
  \indent {\rm (3)} There exists a non-degenerate $G$-invariant bilinear form $b: V \times V \to \mathbb{C}$. \\
If one of the above equivalent statements holds, then such a bilinear form $b$ is unique up to scalar multiples and satisfies $b(w, v) = \nu_2(V) \cdot b(v, w)$ for all $v, w \in V$. In other words, $b$ is symmetric if $\nu_2(V) = +1$ and is skew-symmetric if $\nu_2(V) = -1$.
\end{theorem*}

For $n \ge 3$, the representation-theoretic meaning of the $n$-th FS indicator is less obvious than the second one and there is no such theorem involving the $n$-th FS indicator. Hence, the second FS indicator could be of special interest. Unless otherwise noted, we simply call $\nu_2$ the {\em FS indicator} and refer to $\nu_n$ for $n \ge 3$ as {\em higher FS indicators}.

Generalizing those for compact groups, the FS indicator and higher ones have been defined for various algebraic objects, including (quasi-)Hopf algebras, tensor categories and conformal field theories; see \cite{MR1808131,MR2104908,MR2095575,MR1436801,MR1657800,MR2029790,MR2313527,MR2381536,MR2366965}. Among others, the theory of Ng and Schauenburg \cite{MR2313527,MR2381536,MR2366965} is especially important since it gives a unified category-theoretical understanding of all of \cite{MR1808131,MR2104908,MR2095575,MR1436801,MR1657800,MR2029790}. For the case of a semisimple (quasi-)Hopf algebra, a generalization of the Frobenius--Schur theorem is also formulated and proved; see \cite{MR1808131,MR2104908,MR2095575}. These results have many applications in Hopf algebras and tensor categories; see \cite{MR2725181,MR1942273,MR1919158,MR2213320,KMN09,MR2196640,MR2774703,KenichiShimizu:2012}. 

On the other hand, there are several generalizations in other directions. For example, the earlier result of Linchenko and Montgomery \cite{MR1808131} can be thought, in fact, as a generalization of the Frobenius--Schur theorem for a finite-dimensional semisimple algebra with an anti-algebra involution. Based on their result, Hanaki and Terada \cite{TeradaJunya:2006-03} proved a generalization of the Frobenius--Schur theorem for association schemes and gave some applications to association schemes. Doi \cite{MR2674691} reconstructed the results of \cite{MR1808131} with an emphasis on the use of the theory of symmetric algebras. Recently, Geck \cite{2011arXiv1110.5672G} proved a result similar to Doi and gave some applications to finite Coxeter groups.

Unlike Hopf algebras, the representation categories of such algebras do not have a natural structure of a monoidal category and therefore we cannot understand these results in the framework of Ng and Schauenburg. Our first question is:

\begin{question}
  \label{Q:motivation-1}
  Is there a good category-theoretical framework to understand the FS indicator and the Frobenius--Schur theorem for such algebras?
\end{question}

The second question is about the twisted versions of some of the above. Given an automorphism $\tau$ of a finite group (or a semisimple Hopf algebra), the $n$-th $\tau$-twisted FS indicator $\nu_n^\tau$ is defined by twisting the definition of $\nu_n$ by $\tau$ and the twisted version of the Frobenius--Schur theorem is also formulated and proved; see \cite{MR1078503,MR2068079,MR2879228}. Our second question is:

\begin{question}
  \label{Q:motivation-2}
  If there is an answer to Question~\ref{Q:motivation-1}, then what is its twisted version?
\end{question}

In this paper, we give answers to these questions. Following the approaches of \cite{MR1657800,MR2029790,MR2313527}, we see that the duality is what we really need to define the FS indicator. This observation leads us to the notion of a {\em category with duality}, which has been well-studied in the theory of Witt groups \cite{MR2181829,MR2520968}. As an answer to Question~\ref{Q:motivation-1}, we introduce and study the FS indicator for categories with duality. Considering a suitable category and suitable duality, we can recover various generalizations of the Frobenius--Schur theorem for compact groups. We also introduce a method to ``twist'' the given duality on a category. This gives an answer to Question~\ref{Q:motivation-2}; in fact, the twisted FS indicator can be understood as the ``untwisted'' FS indicator with respect to the twisted duality.


\subsection{Organization and Summary of Results}

The present paper is organized as follows: In Section~\ref{sec:categ-with-dual}, following Mac Lane \cite{MR1712872}, we recall some basic results on adjoint functors and then introduce a category with duality in terms of adjunctions. By generalizing the definitions of  \cite{MR1657800,MR2029790,MR2313527}, we define the FS indicator of an object of a category with duality over a field $k$ (Definition~\ref{def:FS-ind}). We also introduce a general method to ``twist'' the given duality by an adjunction. Then the twisted FS indicator is defined to be the ``untwisted'' FS indicator with respect to the twisted duality.

Pivotal Hopf algebras are introduced as a class of Hopf algebras whose representation category is a pivotal monoidal category; see, e.g., \cite{KMN09}. Motivated by this notion, in Section~\ref{sec:piv-alg}, a {\em pivotal algebra} is defined to be a triple $(A, S, g)$ consisting of an algebra $A$, an anti-algebra map $S: A \to A$ and an invertible element $g \in A$ satisfying certain conditions (Definition~\ref{def:piv-alg}). The representation category of a pivotal algebra is not monoidal in general but has duality. Therefore the FS indicator of an $A$-module is defined in the way of Section~\ref{sec:categ-with-dual}. In Section~\ref{sec:piv-alg}, we study the FS indicator for pivotal algebras and prove some fundamental properties of them.

From our point of view, \eqref{eq:FS-formula-cpt} is not the definition but a formula to compute the FS indicator. It is natural to ask when such a formula exists. In Section~\ref{sec:piv-alg}, we also give a formula for a separable pivotal algebra (Theorem~\ref{thm:piv-FS-ind-ch}); if a pivotal algebra $A = (A, S, g)$ has a separability idempotent $E$, then
\begin{equation}
  \label{eq:FS-formula-intro}
  \nu(V) = \sum_i \chi_V \Big( S(E'_i) g E''_i \Big)
  \quad \Big( E = \sum_i E_i' \otimes E_i'' \Big)
\end{equation}
for all finite-dimensional left $A$-module $V$, where $\chi_V$ is the character of $V$. The relation between this formula and the results of \cite{MR1808131,MR2674691,2011arXiv1110.5672G} is discussed. By specializing \eqref{eq:FS-formula-intro}, we obtain a formula for group-like algebras (Example \ref{ex:FS-ind-ch-GLalg}) and for finite-dimensional weak Hopf $C^*$-algebras (Example \ref{ex:FS-ind-ch-WHA}). We also obtain the formula of Mason and Ng \cite{MR2104908} for finite-dimensional semisimple quasi-Hopf algebras and its twisted version (\S\ref{subsec:quasi-hopf-algebras}).

In Section~\ref{sec:coalgebras}, we introduce a {\em copivotal coalgebra} as the dual notion of pivotal algebras. Each result of Section~\ref{sec:piv-alg} has an analogue in the case of copivotal coalgebras. A crucial difference from the case of algebras is that there are infinite-dimensional coseparable coalgebras. For example, the Hopf algebra $R(G)$ of continuous representative functions on a compact group $G$ is coseparable with coseparability idempotent given by the Haar measure on $G$. The Formula~\eqref{eq:FS-formula-cpt} is obtained by applying the coalgebraic version of~\eqref{eq:FS-formula-intro} to $R(G)$.

In Section~\ref{sec:quantum-sl_2}, we apply our results to the quantum coordinate algebra $\mathcal{O}_q(SL_2)$ and the quantized universal enveloping algebra $U_q(\mathfrak{sl}_2)$. We first determine the FS indicator of all simple right $\mathcal{O}_q(SL_2)$-comodules. In a similar way, we also determine the twisted FS indicator with respect to an involution of $\mathcal{O}_q(SL_2)$ corresponding to the group homomorphism
\begin{equation*}
  SL_2(k) \to SL_2(k), \quad
  \begin{pmatrix} a & b \\ c & d \end{pmatrix} \mapsto \begin{pmatrix} a & - b \\ - c & d \end{pmatrix}
\end{equation*}
in the classical limit $q \to 1$. Similar results for $U_q(\mathfrak{sl}_2)$ are also given by using the Hopf pairing between $\mathcal{O}_q(SL_2)$ and $U_q(\mathfrak{sl}_2)$.

\begin{remark}
  To answer Question~\ref{Q:motivation-1}, we need to work in a ``non-monoidal'' setting. Since, as we have remarked, the Frobenius--Schur theory has some good applications even in non-monoidal settings, the FS indicator for categories with duality could be interesting. However, to define the higher (twisted) FS indicators, a monoidal structure seems to be necessary. At least, there is a reason why we cannot define higher FS indicators for categories with duality; see Remarks \ref{rem:transposition-and-E-map} and \ref{rem:FS-higher}.

  It is interesting to construct a twisted version of \cite{MR2313527,MR2381536,MR2366965}. One of the referees kindly pointed out to the author that in May 2012, Daniel Sage gave a talk on a category-theoretic definition of the higher twisted FS indicators for Hopf algebras at the Lie Theory Workshop held at University of Southern California. Independently, after the submission of the first version of this paper, the author obtained a description of the higher twisted FS indicator for Hopf algebras by using the crossed product monoidal category. In this paper, we also mention higher twisted FS indicators for pivotal monoidal categories but leave the details for future work; see \S\ref{subsec:group-action-pivotal}.
\end{remark}

\begin{remark}
  Linchenko and Montgomery \cite{MR1808131} established a relation between the FS indicator and invariant bilinear forms on an irreducible representation. Unlike the case of compact groups, a relation between the FS indicator and ``reality'' of representations is not known in the case of Hopf algebras; in fact, as remarked in \cite{MR1808131}, the reality of a representation of a Hopf algebra is not defined since, in general, a Hopf algebra does not have a good basis like the group elements of the group algebra. In a forthcoming paper, we will introduce the notions of real, complex and quaternionic representations of a Hopf $*$-algebra and Formulate \eqref{eq:FS-ind-RCH} in a Hopf algebraic context. We will also provide an exact quantum analog of the Frobenius--Schur theorem for compact quantum groups.
\end{remark}

\subsection{Notation}

Given a category $\mathcal{C}$ and  $X, Y \in \mathcal{C}$, we denote by $\Hom_\mathcal{C}(X, Y)$ the set of all morphisms from $X$ to $Y$. $\mathcal{C}^{\op}$ means the opposite category of $\mathcal{C}$. An object $X \in \mathcal{C}$ is often written as $X^{\op}$ when it is regarded an object of $\mathcal{C}^{\op}$. A similar notation is used for morphisms. A functor $F: \mathcal{C} \to \mathcal{D}$ is denoted by $F^{\op}$ if it is regarded as a functor $\mathcal{C}^{\op} \to \mathcal{D}^{\op}$.

Throughout, we work over a fixed field $k$ whose characteristic is not two. By an algebra, we mean a unital associative algebra over $k$. Given a vector space $V$ (over $k$), we denote by $V^\vee = \Hom_k(V, k)$ the dual space of $V$. For $f \in V^\vee$ and $v \in V$, we often write $f(v)$ as $\langle f, v \rangle$. Unless otherwise noted, the unadorned tensor symbol $\otimes$ means the tensor product over $k$. Given $t \in V^{\otimes n}$, we often write $t$ as
\begin{equation*}
  t = t^{1} \otimes t^{2} \otimes \dotsb \otimes t^{n} \in V \otimes V \otimes \dotsb \otimes V
\end{equation*}
The comultiplication and the counit of a coalgebra are denoted by $\Delta$ and $\varepsilon$, respectively. For an element $c$ of a coalgebra, we use Sweedler's notation
\begin{equation*}
  \Delta(c) = c_{(1)} \otimes c_{(2)}, \quad
  \Delta(c_{(1)}) \otimes c_{(2)} = c_{(1)} \otimes c_{(2)} \otimes c_{(3)} = c_{(1)} \otimes \Delta(c_{(2)}),
  \ \dotsc
\end{equation*}

\section{Categories with Duality}

\label{sec:categ-with-dual}
\subsection{Adjunctions}

Following Mac Lane \cite{MR1712872}, we recall basic results on adjunctions. Let $\mathcal{C}$ and $\mathcal{D}$ be categories. An {\em adjunction} from $\mathcal{C}$ to $\mathcal{D}$ is a triple $(F, G, \Phi)$ consisting of functors $F: \mathcal{C} \to \mathcal{D}$ and $G: \mathcal{D} \to \mathcal{C}$ and a natural bijection $\Phi_{X,Y}: \Hom_\mathcal{D}(F(X), Y) \to \Hom_\mathcal{C}(X, G(Y))$ ($X \in \mathcal{C}$, $Y \in \mathcal{D}$).

Given an adjunction $(F, G, \Phi)$ from $\mathcal{C}$ to $\mathcal{D}$, the {\em unit} $\eta: \id_{\mathcal{C}} \to G F$ and the {\em counit} $\varepsilon: F G \to \id_{\mathcal{D}}$ of the adjunction $(F, G, \Phi)$ are defined by
$\eta_X^{} = \Phi_{X,F(X)}^{}(\id_{F(X)}^{})$ and $\varepsilon_Y^{} = \Phi_{G(Y),Y}^{-1} (\id_{G(Y)}^{})$ for $X \in \mathcal{C}$ and $Y \in \mathcal{D}$, respectively. They satisfy the counit-unit equations
\begin{equation}
  \label{eq:adjunction-2}
  \varepsilon_{F(X)} \circ F(\eta_{X}^{}) = \id_{F(X)} \text{\quad and \quad}
  G(\varepsilon_{Y}) \circ \eta_{G(Y)}^{} = \id_{G(Y)}
\end{equation}
for all $X \in \mathcal{C}$ and $Y \in \mathcal{D}$. By using $\eta$, the natural bijection $\Phi$ is expressed as
\begin{equation}
  \label{eq:adjunction-3}
  \Phi_{X,Y}(f) = G(f) \circ \eta_{X}
  \quad (f \in \Hom_\mathcal{C}(F(X), Y)
\end{equation}
Similarly, by using $\varepsilon$, the inverse of $\Phi$ is expressed as
\begin{equation}
  \label{eq:adjunction-4}
  \Phi_{X,Y}^{-1}(g) = \varepsilon_{Y} \circ F(g)
  \quad (g \in \Hom_\mathcal{D}(X, G(Y))
\end{equation}
Note that $\circ$ at the right-hand side stands for the composition in $\mathcal{D}$. We will deal with the case where $\mathcal{D} = \mathcal{C}^{\op}$, the opposite category of $\mathcal{C}$.

Each adjunction is determined by its unit and counit; indeed, let $F: \mathcal{C} \to \mathcal{D}$ and $G: \mathcal{D} \to \mathcal{C}$ be functors, and let $\eta: \id_{\mathcal{C}} \to G F$ and $\varepsilon: F G \to \id_{\mathcal{D}}$ be natural transformations satisfying \eqref{eq:adjunction-2}. If we define $\Phi$ by \eqref{eq:adjunction-3}, then the triple $(F, G, \Phi)$ is an adjunction from $\mathcal{C}$ to $\mathcal{D}$ whose unit and counit are $\eta$ and $\varepsilon$. From this reason, we abuse terminology and refer to such a quadruple $(F, G, \eta, \varepsilon)$ as an adjunction from $\mathcal{C}$ to $\mathcal{D}$.

\subsection{Categories with Duality}
\label{subsec:cat-with-duality}

The following terminologies are taken from Balmer \cite{MR2181829} and Calm\'es-Hornbostel \cite{MR2520968}.

\begin{definition}
  \label{def:cat-with-duality}
  A {\em category with duality} is a triple $\mathcal{C} = (\mathcal{C}, (-)^\vee, j)$ consisting of a category $\mathcal{C}$, a contravariant functor $(-)^\vee: \mathcal{C} \to \mathcal{C}$ and a natural transformation $j: \id_\mathcal{C} \to (-)^{\vee \vee}$ satisfying
  \begin{equation}
    \label{eq:duality-1}
    (j_X)^\vee \circ j_{X^\vee} = \id_{X^\vee}
  \end{equation}
  for all $X \in \mathcal{C}$. If, moreover, $j$ is a natural isomorphism, then we say that $\mathcal{C}$ is a {\em category with strong duality}, or, simply, $\mathcal{C}$ is {\em strong}.
\end{definition}

Let $\mathcal{C}$ be a category with duality. We call the functor $(-)^\vee: \mathcal{C} \to \mathcal{C}$ the {\em duality functor} of $\mathcal{C}$. A pivotal monoidal category is an example of categories with duality; see \cite[Appendix]{MR2095575}. Thus we call the natural transformation $j: \id_\mathcal{C} \to (-)^{\vee\vee}$ the {\em pivotal morphism} of $\mathcal{C}$.

\begin{example}
  \label{ex:cat-w-dual-invol-Hopf}
  Let $H$ be a Hopf algebra with antipode $S$. If $H$ is {\em involutory}, {\em i.e.}, $S^2 = \id_H$, then the category $\Mod(H)$ of left $H$-modules is a category with duality; the duality functor is given by taking the dual $H$-module and the pivotal morphism is given by the canonical map
  \begin{equation}
    \label{eq:k-Vec-piv-mor}
    \iota_V: V \to V^{\vee\vee} = \Hom_k(\Hom_k(V, k),k),
    \quad \langle \iota(v), f \rangle = \langle f, v \rangle
    \quad (v \in V, f \in V^\vee)
  \end{equation}
  The full subcategory $\fdMod(H)$ of $\Mod(H)$ of finite-dimensional left $H$-modules is a category with strong duality since $\iota_V$ is an isomorphism if (and only if) $\dim_k V < \infty$.
\end{example}

Let $D$ denote the duality functor of $\mathcal{C}$ regarded as a covariant functor from $\mathcal{C}$ to $\mathcal{C}^{\op}$. Definition \ref{def:cat-with-duality} says that the quadruple $(D, D^{\op}, j, j^{\op}): \mathcal{C} \to \mathcal{C}^{\op}$ is an adjunction. Hence we obtain a natural bijection
\begin{equation}
  \label{eq:transposition-1}
  \begin{aligned}
    \Trans_{X,Y}:
    \Hom_\mathcal{C}(X, Y^\vee)
    & = \Hom_{\mathcal{C}^{\op}}(D(Y), X^{\op}) \\
    \longrightarrow & \Hom_\mathcal{C}(Y, D^{\op}(X^{\op}))
    = \Hom_\mathcal{C}(Y, X^\vee)
    \quad (X, Y \in \mathcal{C})
  \end{aligned}
\end{equation}
which we call the {\em transposition map}. By~\eqref{eq:adjunction-3}, $\Trans_{X,Y}$ is expressed as
\begin{equation}
  \label{eq:transposition-2}
  \Trans_{X,Y}(f) = f^\vee \circ j_Y
  \quad (f \in \Hom_\mathcal{C}(X, Y^\vee))
\end{equation}
By \eqref{eq:adjunction-4}, we have $\Trans_{X,Y}^{-1}(g) = g^\vee \circ j_X = \Trans_{Y,X}(g)$ for $g \in \Hom_\mathcal{C}(Y, X^\vee)$. Hence we have
\begin{equation}
  \label{eq:transposition-3}
  \Trans_{Y,X} \circ \Trans_{X,Y} = \id_{\Hom_\mathcal{C}(X,Y^\vee)}^{}
\end{equation}

Note that $j$ is not necessarily an isomorphism. By understanding a category with duality as a kind of adjunction, we obtain the following characterization of categories with strong duality.

\begin{lemma}
  \label{lem:duality-strong}
  For a category $\mathcal{C}$ with duality, the following are equivalent: \\
  \indent {\rm (1)} $\mathcal{C}$ is a category with strong duality. \\
  \indent {\rm (2)} The duality functor $(-)^\vee: \mathcal{C} \to \mathcal{C}^{\op}$ is an equivalence.
\end{lemma}
\begin{proof}
  Let, in general, $(F, G, \eta, \varepsilon)$ be an adjunction between some categories. Then $F$ is fully faithful if and only if $\eta$ is an isomorphism \cite[IV.3]{MR1712872}. Now we apply this result to the above quadruple $(D, D^{\op}, j, j^{\op})$ as follows: If $\mathcal{C}$ is strong, then $D$ is fully faithful. Since $X \cong X^{\vee\vee} = D(X^\vee)$ ($X \in \mathcal{C}$), $D$ is essentially surjective. Hence, $D$ is an equivalence. The converse is clear, since an equivalence of categories is fully faithful.
\end{proof}

Following \cite{MR2520968}, we introduce duality preserving functors and related notions:

\begin{definition}
  Let $\mathcal{C}$ and $\mathcal{D}$ be categories with duality. A {\em duality preserving functor} from $\mathcal{C}$ to $\mathcal{D}$ is a pair $(F, \xi)$ consisting of a functor $F: \mathcal{C} \to \mathcal{D}$ and a natural transformation $\xi: F(X^\vee) \to F(X)^\vee$ ($X \in \mathcal{C}$) making
  \begin{equation}
    \label{eq:duality-2}
    \begin{CD}
      F(X) @>{F(j_X)}>> F(X^{\vee\vee}) \\
      @V{j_{F(X)}}VV @VV{\xi_{X^\vee}}V \\
      F(X)^{\vee\vee} @>>{\xi_X^{\vee}}> F(X^\vee)^\vee
    \end{CD}
  \end{equation}
  commute for all $X \in \mathcal{C}$. If, moreover, $\xi$ is an isomorphism, then $(F, \xi)$ is said to be a {\em strong}. If $\xi$ is the identity, then $(F, \xi)$ is said to be {\em strict}.

  Now let $(F, \xi), (G, \zeta): \mathcal{C} \to \mathcal{D}$ be such functors. A {\em morphism of duality preserving functors} from $(F, \xi)$ to $(G, \zeta)$ is a natural transformation $h: F \to G$ making
  \begin{equation*}
    \begin{CD}
      F(X^\vee) @>{\xi_X}>> F(X)^\vee \\
      @V{h_{X^\vee}}VV @AA{h_X^\vee}A \\
      G(X^\vee) @>>{\zeta_X}> G(X)^\vee
    \end{CD}
  \end{equation*}
  commute for all $X \in \mathcal{C}$.
\end{definition}

If $(F, \xi): \mathcal{C} \to \mathcal{D}$ and $(G, \zeta): \mathcal{D} \to \mathcal{E}$ are duality preserving functors between categories with duality, then the composition $G \circ F: \mathcal{C} \to \mathcal{E}$ becomes a duality preserving functor with
\begin{equation*}
  \begin{CD}
    G(F(X^\vee)) @>{G(\xi_X)}>> G(F(X)^\vee) @>{\zeta_{F(X)}}>> G(F(X))^\vee
    @. \quad (X \in \mathcal{C})
  \end{CD}
\end{equation*}
One can check that categories with duality form a 2-category; 1-arrows are duality preserving functors and 2-arrows are morphisms of duality preserving functors. Hence we can define an {\em isomorphism} and an {\em equivalence} of categories with duality in the usual way.

Given a duality preserving functor $(F, \xi): \mathcal{C} \to \mathcal{D}$, we define
\begin{equation*}
  \tilde{F}_{X,Y}: \Hom_\mathcal{C}(X, Y^\vee) \to \Hom_\mathcal{D}(F(X), F(Y)^\vee)
  \quad (X, Y \in \mathcal{C})
\end{equation*}
by $\tilde{F}_{X,Y}(f) = \xi_Y \circ F(f)$ for $f: X \to Y^\vee$. $\tilde{F}$ is compatible with the transposition map in the sense that the diagram
\begin{equation}
  \label{eq:transposition-4}
  \begin{CD}
    \Hom_\mathcal{C}(X, Y^\vee) @>{\tilde{F}_{X,Y}}>>
    \Hom_\mathcal{D}(F(X), F(Y)^\vee) \\
    @V{\Trans_{X,Y}}VV @VV{\Trans_{F(X),F(Y)}}V \\
    \Hom_\mathcal{C}(Y, X^\vee) @>>{\tilde{F}_{X,Y}}>
    \Hom_\mathcal{D}(F(Y), F(X)^\vee)
  \end{CD}
\end{equation}
commutes for all $X, Y \in \mathcal{C}$. Indeed, we have
\begin{align*}
  & \Trans_{F(X),F(Y)} (\tilde{F}_{X,Y}(f))
  = (\xi_Y \circ F(f))^\vee \circ j_{F(Y)}
  = F(f)^\vee \circ \xi_Y^\vee \circ j_{F(Y)} \\
  & \qquad = F(f)^\vee \circ \xi_{X^\vee} \circ {F(j_X)}
  = \xi_Y \circ F(f^\vee) \circ F(j_X)
  = \tilde{F}_{X,Y} (\Trans_{X,Y}(f))
\end{align*}

Now suppose that $\mathcal{C}$ is a category with strong duality. Then:

\begin{lemma}
  \label{lem:duality-op}
  $\mathcal{C}^{\op}$ is a category with duality with the same duality functor as $\mathcal{C}$ and pivotal morphism $(j^{-1})^{\op}$. The duality functor on $\mathcal{C}$ is an equivalence of categories with duality between $\mathcal{C}$ and $\mathcal{C}^{\op}$.
\end{lemma}

Hence, from~\eqref{eq:transposition-4} with $(F, \xi) = ((-)^\vee, \id_{(-)^\vee}): \mathcal{C}^{\op} \to \mathcal{C}$, we see that
\begin{equation}
  \label{eq:transposition-5}
  \begin{CD}
    \Hom_{\mathcal{C}}(X^\vee, Y) @=
    \Hom_{\mathcal{C}^{\op}}(Y^{\op}, (X^\vee)^{\op}) @>{(-)^\vee}>>
    \Hom_{\mathcal{C}}(Y^\vee, X^{\vee\vee}) \\
    @V{\Trans_{X,Y}^{\op}}VV @. @VV{\Trans_{Y^\vee, X^\vee}}V \\
    \Hom_{\mathcal{C}}(Y^\vee, X) @=
    \Hom_{\mathcal{C}^{\op}}(X^{\op}, (Y^\vee)^{\op}) @>>{(-)^\vee}>
    \Hom_{\mathcal{C}}(X^\vee, Y^{\vee\vee})
  \end{CD}
\end{equation}
commutes for all $X, Y \in \mathcal{C}$, where $\Trans_{X,Y}^{\op}$ is the transposition map for $\mathcal{C}^{\op}$ regarded as a map $\Hom_{\mathcal{C}}(X^\vee, Y) \to \Hom_{\mathcal{C}}(Y^\vee, X)$. Explicitly, it is given by $\Trans_{X,Y}^{\op}(f) = j_X^{-1} \circ f^\vee$.

\begin{remark}
  \label{rem:transposition-and-E-map}
  If $\mathcal{C}$ is a pivotal monoidal category, then there is a natural bijection
  \begin{equation*}
    \Hom_\mathcal{C}(X^{\vee}, Y) \cong \Hom_\mathcal{C}(\mathbf{1}, X \otimes Y)
    \quad (X, Y \in \mathcal{C}),
  \end{equation*}
  where $\otimes$ is the tensor product of $\mathcal{C}$ and $\mathbf{1} \in \mathcal{C}$ is the unit object. The diagram
  \begin{equation*}
    \begin{CD}
      \Hom_\mathcal{C}(X^{\vee},X) @>{\cong}>> \Hom_\mathcal{C}(\mathbf{1}, X \otimes X) \\
      @V{\Trans^{\op}_{X,X}}VV @VV{E_X^{(2)}}V \\
      \Hom_\mathcal{C}(X^{\vee},X) @>>{\cong}> \Hom_\mathcal{C}(\mathbf{1}, X \otimes X)
    \end{CD}
  \end{equation*}
  commutes, where the horizontal arrows are the canonical bijections and $E_X^{(2)}$ is the map used in \cite{MR2381536} to define the FS indicator.
\end{remark}

\subsection{Frobenius--Schur Indicator}

Recall that a category $\mathcal{C}$ is said to be {\em $k$-linear} if each hom-set is a vector space over $k$ and the composition of morphisms is $k$-bilinear. A functor $F: \mathcal{C} \to \mathcal{D}$ between $k$-linear categories is said to be {\em $k$-linear} if the map $\Hom_\mathcal{C}(X, Y) \to \Hom_\mathcal{D}(F(X), F(Y))$, $f \mapsto F(f)$ is $k$-linear for all $X, Y \in \mathcal{C}$. Note that $\mathcal{C}^{\op}$ is $k$-linear if $\mathcal{C}$ is. Thus the $k$-linearity of a contravariant functor makes sense.

\begin{definition}
  By a {\em category with duality over $k$}, we mean a $k$-linear category with duality whose duality functor is $k$-linear.
\end{definition}

For simplicity, in this section, we always assume that a category $\mathcal{C}$ with duality over $k$ satisfies the following finiteness condition:
\begin{equation}
  \label{eq:assumption-fin-dim}
  \dim_k \Hom_\mathcal{C}(X, Y) < \infty \quad \text{for all $X, Y \in \mathcal{C}$}
\end{equation}

\begin{definition}
  \label{def:FS-ind}
  Let $\mathcal{C}$ be a category with duality over $k$. The {\em Frobenius--Schur indicator} (or FS indicator, for short) of $X \in \mathcal{C}$ is defined and denoted by $\nu(X) = \Trace(\Trans_{X,X})$, where $\Trace$ means the trace of a linear map.
\end{definition}

The following is a list of basic properties of the FS indicator:

\begin{proposition}
  \label{prop:FS-ind-basic}
  Let $\mathcal{C}$ be a category with duality over $k$ and let $X \in \mathcal{C}$.

  {\rm (a)} $\nu(X)$ depends on the isomorphism class of $X \in \mathcal{C}$.

  {\rm (b)} $\nu(X) = \dim_k B^+_\mathcal{C}(X) - \dim_k B_\mathcal{C}^-(X)$, where
  \begin{equation*}
    B_\mathcal{C}^{\pm}(X) = \{ b: X \to X^\vee \mid \Trans_{X,X}(b) = \pm b \}.
  \end{equation*}

  {\rm (c)} Let $X_1, X_2 \in \mathcal{C}$. If their biproduct $X_1 \oplus X_2$ exists, then we have
  \begin{equation*}
    \nu(X_1 \oplus X_2) = \nu(X_1) + \nu(X_2).
  \end{equation*}
\end{proposition}
\begin{proof}
  (a) Let $p: X \to Y$ be an isomorphism in $\mathcal{C}$. Then
  \begin{equation*}
    \Hom_\mathcal{C}(p,p^\vee): \Hom_\mathcal{C}(Y, Y^\vee) \to \Hom_\mathcal{C}(X, X^\vee),
    \quad f \mapsto p^\vee \circ f \circ p
  \end{equation*}
  is an isomorphism. By the naturality of the transposition map, the diagram
  \begin{equation*}
    \begin{CD}
      \Hom_\mathcal{C}(Y, Y^\vee) @>{\Trans_{Y,Y}}>> \Hom_\mathcal{C}(Y, Y^\vee) \\
      @V{\Hom_\mathcal{C}(p,p^\vee)}VV @VV{\Hom_\mathcal{C}(p,p^\vee)}V \\
      \Hom_\mathcal{C}(X, X^\vee) @>>{\Trans_{X,X}}> \Hom_\mathcal{C}(X, X^\vee)
    \end{CD}
  \end{equation*}
  commutes. Hence, we have $\nu(X) = \Trace(\Trans_{X,X}) = \Trace(\Trans_{Y,Y}) = \nu(Y)$.

  (b) The result follows from~\eqref{eq:transposition-3} and the fact that the trace of an operator is the sum of its eigenvalues.

  (c) For $a = 1, 2$, let $i_a: X_a \to X_1 \oplus X_2$ and $p_a: X_1 \oplus X_2 \to X_a$ be the inclusion and the projection, respectively. For $a, b, c, d = 1, 2$, we set
  \begin{equation*}
    \Trans_{a b}^{c d}
    = p_{c d} \circ \Trans_{X_1 \oplus X_2, X_1 \oplus X_2} \circ i_{a b}:
    \Hom_\mathcal{C}(X_a, X_b^\vee) \to \Hom_\mathcal{C}(X_c, X_d^\vee)
  \end{equation*}
  where $i_{a b} = \Hom_\mathcal{C}(p_a^{}, p_b^\vee)$ and $p_{c d} = \Hom_\mathcal{C}(i_c^{}, i_d^\vee)$. By linear algebra, we have
  \begin{equation}
    \label{eq:FS-ind-additive-proof-1}
    \Trace(\Trans_{X_1 \oplus X_2, X_1 \oplus X_2})
    = \Trace(\Trans_{1 1}^{1 1}) + \Trace(\Trans_{1 2}^{1 2}) + \Trace(\Trans_{2 1}^{2 1}) + \Trace(\Trans_{2 2}^{2 2})
  \end{equation}
  Now, by the naturality of the transposition map, we compute
  \begin{align*}
    \Trans_{a b}^{a b}
    & = \Hom_\mathcal{C}(i_a, i_b^\vee) \circ \Trans_{X_1 \oplus X_2, X_1 \oplus X_2}
    \circ \Hom_\mathcal{C}(p_a^{}, p_b^\vee) \\
    & = \Hom_\mathcal{C}(i_a, i_b^\vee)
    \circ \Hom_\mathcal{C}(p_b^{}, p_a^\vee) \circ \Trans_{X_a, X_b} \\
    & = \Hom_{\mathcal{C}}(p_b^{} \circ i_a^{}, p_a^\vee \circ i_b^\vee) \circ \Trans_{X_a, X_b}
  \end{align*}
  Hence $\Trans_{a b}^{a b}$ is equal to $\Trans_{X_a, X_a}$ if $a = b$ and is zero if otherwise. Combining this result with~\eqref{eq:FS-ind-additive-proof-1}, we obtain $\nu(X_1 \oplus X_2) = \nu(X_1) + \nu(X_2)$.
\end{proof}

The FS indicator is an invariant of categories with duality over $k$. Indeed, the commutativity of \eqref{eq:transposition-4} yields the following proposition:

\begin{proposition}
  \label{prop:FS-ind-inv}
  Let $(F, \xi): \mathcal{C} \to \mathcal{D}$ be a strong duality preserving functor. If $F$ is $k$-linear and fully faithful, then we have $\nu(F(X)) = \nu(X)$ for all $X \in \mathcal{C}$.
\end{proposition}

Similarly, we obtain the following proposition from~\eqref{eq:transposition-5}:

\begin{proposition}
  \label{prop:FS-ind-dual}
  Suppose that $\mathcal{C}$ is a category with strong duality over $k$. Then, for all $X \in \mathcal{C}$, we have $\nu(X^\vee) = \Trace(\Trans^\op_{X,X}) = \nu(X^{\op})$.
\end{proposition}

Let $\mathcal{A}$ be a $k$-linear Abelian category. Recall that a nonzero object of $\mathcal{A}$ is said to be {\em simple} if it has no proper subobjects. We say that a simple object $V \in \mathcal{A}$ is {\em absolutely simple} if $\End_{\mathcal{A}}(V) \cong k$. Note that the opposite category $\mathcal{A}^{\op}$ is also $k$-linear and Abelian. It is easy to see that an object of $\mathcal{A}$ is (absolutely) simple if and only if it is (absolutely) simple as an object of $\mathcal{A}^{\op}$.

\begin{proposition}
  \label{prop:FS-ind-abs-simple}
  Let $\mathcal{C}$ be an Abelian category with strong duality over $k$, and let $X \in \mathcal{C}$. \\
  \indent {\rm (a)} If $X$ is a finite biproduct of simple objects, then $\nu(X) = \nu(X^{\vee})$. \\
  \indent {\rm (b)} If $X$ is absolutely simple, then $\nu(X) \in \{ 0, \pm 1 \}$. $\nu(X) \ne 0$ if and only if $X$ is self-dual, that is, $X$ is isomorphic to $X^{\vee}$.
\end{proposition}
\begin{proof}
  (a) We first claim that if $V \in \mathcal{C}$ is simple, then $\nu(V) = \nu(V^{\vee})$. Let $V \in \mathcal{C}$ be a simple object. Since $(-)^\vee: \mathcal{C} \to \mathcal{C}^{\op}$ is an equivalence, $V^\vee$ is simple as an object of $\mathcal{C}^{\op}$ and hence it is simple as an object of $\mathcal{C}$. If $V$ is isomorphic to $V^{\vee}$, then our claim is obvious. Otherwise, we have
  \begin{equation*}
    \Hom_\mathcal{C}(V, V^{\vee}) = 0 \text{\quad and \quad}
    \Hom_\mathcal{C}(V^\vee, V^{\vee\vee}) \cong \Hom_\mathcal{C}(V^{\vee}, V) = 0
  \end{equation*}
  by Schur's lemma. Therefore $\nu(V) = 0 = \nu(V^{\vee})$ follows.

  Now write $X$ as $X = V_1 \oplus \dotsb \oplus V_m$ for some simple objects $V_1, \dotsc, V_m \in \mathcal{C}$. By the above arguments and the additivity of the FS indicator, we have
  \begin{equation*}
    \nu(X^{\vee}) = \nu(V_1^{\vee}) + \dotsb + \nu(V_m^{\vee}) = \nu(V_1) + \dotsb + \nu(V_m) = \nu(X)
  \end{equation*}

  (b) Suppose that $X \in \mathcal{C}$ is absolutely simple. If $X$ is isomorphic to $X^{\vee}$, then
  \begin{equation*}
    \dim_k \Hom_\mathcal{C}(X, X^{\vee}) = \dim_k \End_{\mathcal{C}}(X) = 1
  \end{equation*}
  and hence $\nu(X)$ is either $+1$ or $-1$ by Proposition~\ref{prop:FS-ind-basic} (b). Otherwise, $\nu(X) = 0$ as we have seen in the proof of (a).
\end{proof}

\begin{remark}
  \label{rem:FS-higher}
  If $\mathcal{C}$ is a pivotal monoidal category over $k$, then the $n$-th FS indicator $\nu_n(X)$ of $X \in \mathcal{C}$ is defined for each integer $n \ge 2$; see \cite{MR2381536}. The commutativity of \eqref{eq:transposition-4} implies $\nu_2(X) = \nu(X^\vee)$ (see also Remark~\ref{rem:transposition-and-E-map}). However, in view of Proposition~\ref{prop:FS-ind-abs-simple}, $\nu(X) = \nu_2(X)$ always holds in the case where $\mathcal{C}$ is strong, Abelian, and semisimple. We prefer our Definition~\ref{def:FS-ind} since it is more convenient when we discuss the relation between the FS indicator and invariant bilinear forms.

  One would like to define higher FS indicators for an object of a category with duality over $k$ by extending that for an object of a pivotal monoidal category over $k$. This is impossible because of the following example: For a group $G$, we denote by $\Vect^G_\fdim$ the $k$-linear pivotal monoidal category of finite-dimensional $G$-graded vector spaces over $k$. The $n$-th FS indicator of $V = \bigoplus_{x \in G} V_x \in \Vect^G_\fdim$ is given by
  \begin{equation}
    \label{eq:FS-ind-G-gr}
    \nu_n(V) = \sum_{x \in G[n]} \dim_k(V_x),
    \text{\quad where $G[n] = \{ x \in G \mid x^n = 1 \}$}
  \end{equation}
  Now we put
  \begin{equation*}
    \mathcal{C} = \Vect_\fdim^{\mathbb{Z}_4 \times \mathbb{Z}_4}
    \text{\quad and \quad}
    \mathcal{D} = \Vect_\fdim^{\mathbb{Z}_2 \times \mathbb{Z}_8}
  \end{equation*}
  There exists a bijection $f: \mathbb{Z}_4 \times \mathbb{Z}_4 \to \mathbb{Z}_2 \times \mathbb{Z}_8$ such that $f(x^{-1}) = f(x)^{-1}$ for all $x \in \mathbb{Z}_4 \times \mathbb{Z}_4$. $f$ induces an equivalence $\mathcal{C} \approx \mathcal{D}$ of categories with duality over $k$. If we could define the $n$-th FS indicator for categories with duality over $k$, then there would exist at least one equivalence $F: \mathcal{C} \to \mathcal{D}$ such that $\nu_n(F(X)) = \nu_n(X)$ for all $X \in \mathcal{C}$ and $n \ge 2$. However, by \eqref{eq:FS-ind-G-gr}, there is no such equivalence.
\end{remark}

\subsection{Separable Functors}
\label{subsec:sep-functor}

Let $H$ be an involutory Hopf algebra, e.g., the group algebra of a group $G$. We consider the category $\mathcal{C} = \fdMod(H)$ of Example~\ref{ex:cat-w-dual-invol-Hopf}. The FS indicator $\nu(V)$ of $V \in \mathcal{C}$ is interpreted as follows: Let $\Bil_H(V)$ denote the set of all $H$-invariant bilinear forms on $V$. The transposition map induces
\begin{equation*}
  \Sigma_V: \Bil_H(V) \to \Bil_H(V),
  \quad \Sigma_V(b)(v, w) = b(w, v)
  \quad (b \in \Bil_H(V), v, w \in V)
\end{equation*}
via the canonical isomorphism $\Bil_H(V) \cong \Hom_H(V, V^\vee)$. Now let $\Bil_H^{\pm}(V)$ denote the eigenspace of $\Sigma_V$ with eigenvalue $\pm 1$. Then we have
\begin{equation*}
  \nu(V) = \Trace(\Trans_{V,V}) = \Trace(\Sigma_V) = \dim_k \Bil_H^+(V) - \dim_k \Bil_H^-(V)
\end{equation*}
Hence, from our definition, the relation between $\nu(V)$ and $H$-invariant bilinear forms on $V$ is clear. On the other hand, it is not obvious that $\nu(V)$ is expressed by a formula like \eqref{eq:FS-formula-cpt}. It should be emphasized that, from our point of view, \eqref{eq:FS-formula-cpt} is not the definition of $\nu(V)$ but rather a formula to compute $\nu(V)$. We note that a similar point of view is effectively used to derive a formula of the FS indicator for semisimple finite-dimensional quasi-Hopf algebras in \cite{MR2095575}.

A key notion to derive~\eqref{eq:FS-formula-cpt} is a {\em separable functor} \cite{MR1926102}; a functor $U: \mathcal{C} \to \mathcal{V}$ is said to be {\em separable} if there exists a natural transformation
\begin{equation*}
  \Pi_{X,Y}: \Hom_\mathcal{V}(U(X), U(Y)) \to \Hom_\mathcal{C}(X, Y) \quad (X, Y \in \mathcal{C})
\end{equation*}
such that $\Pi_{X,Y}(U(f)) = f$ for all $f \in \Hom_\mathcal{C}(X, Y)$. Such a natural transformation $\Pi$ is called a {\em section} of $U$. Suppose that $\mathcal{C}$ and $\mathcal{V}$ be $k$-linear. We say that a section $\Pi$ of $U$ is {\em $k$-linear} if $\Pi_{X,Y}$ is $k$-linear for all $X, Y \in \mathcal{C}$.

Now let $\mathcal{C}$ be a category with duality over $k$, and let $\mathcal{V}$ be a $k$-linear category satisfying \eqref{eq:assumption-fin-dim}. A $k$-linear functor $U: \mathcal{C} \to \mathcal{V}$ induces a $k$-linear map
\begin{equation*}
  U_{X,Y}: \Hom_\mathcal{C}(X, Y) \to \Hom_\mathcal{V}(U(X), U(Y))
  \quad f \mapsto U(f)
  \quad (X, Y \in \mathcal{C})
\end{equation*}
If $U$ has a $k$-linear section $\Pi$, we define a linear map
\begin{equation}
  \label{eq:transposition-sep}
  \widetilde{\Trans}_{X,Y}: \Hom_\mathcal{V}(U(X), U(Y^\vee)) \to \Hom_\mathcal{V}(U(Y), U(X^\vee))
  \quad (X, Y \in \mathcal{C})
\end{equation}
so that the diagram
\begin{equation*}
  \begin{CD}
    \Hom_\mathcal{V}(U(X), U(Y^\vee)) @>{\widetilde{\Trans}_{X,Y}}>> \Hom_\mathcal{V}(U(Y), U(X^\vee)) \\
    @V{\Pi_{X,Y^\vee}}VV @AA{U_{Y,X^\vee}}A \\
    \Hom_\mathcal{C}(X, Y^\vee) @>>{\Trans_{X,Y}}> \Hom_\mathcal{C}(Y, X^\vee)
  \end{CD}
\end{equation*}
commutes. By using the well-known identity $\Trace(A B) = \Trace(B A)$, we have
\begin{equation}
  \label{eq:FS-ind-sep}
  \begin{aligned}
    \Trace(\widetilde{\Trans}_{X,X})
    & = \Trace(U_{X,X^\vee} \circ \Trans_{X,X} \circ \Pi_{X, X^\vee}) \\
    & = \Trace(\Pi_{X, X^\vee} \circ U_{X,X^\vee} \circ \Trans_{X,X})
    = \Trace(\Trans_{X,X})
    = \nu(X)
  \end{aligned}
\end{equation}
Before we explain how \eqref{eq:FS-formula-cpt} is derived from \eqref{eq:FS-ind-sep}, we recall the following lemma in linear algebra: Let $f: V \to W$ and $g: W \to V$ be linear maps between finite-dimensional vector spaces. We define
\begin{equation*}
  T: V \otimes W \to V \otimes W,
  \quad T(v \otimes w) = g(w) \otimes f(v).
  \quad (v \in V, w \in W)
\end{equation*}

\begin{lemma}
  \label{lem:trace-1}
  $\Trace(T) = \Trace(f g)$.
\end{lemma}

The dual of this lemma is also useful: Let $B(V, W)$ be the set of bilinear maps $V \times W \to k$ and consider the map $T^\vee: B(V, W) \to B(V, W)$, $T^\vee(b)(v, w) = b(g(w), f(v))$. Since $T^\vee$ is the dual map of $T$ under the identification $B(V, W) \cong (V \otimes W)^\vee$, we have $\Trace(T^\vee) = \Trace(f g)$.

\begin{proof}
  Let $\{ v_i \}$ and $\{ w_j \}$ be a basis of $V$ and $W$, and let $\{ v^i \}$ and $\{ w^j \}$ be the dual basis to $\{ v_i \}$ and $\{ w_j \}$, respectively. Then we have
  \begin{align*}
    \Trace(T)
    = \sum_{i, j} \langle v^i, g(w_j) \rangle \langle w^j, f(v_i) \rangle
    = \sum_j \langle w^j, f(g(w_j)) \rangle = \Trace(f g)
  \end{align*}
  Here, the first and the last equality follow from $\Trace(f) = \sum_i \langle v^i, f(v_i) \rangle$ and the second follows from $x = \sum_i \langle v^i, x \rangle v_i$.
\end{proof}

Now, for simplicity, we assume $k$ to be an algebraically closed field of characteristic zero. Let $H$ be a finite-dimensional semisimple Hopf algebra over $k$,  e.g., the group algebra of a finite group $G$. Then, by the theorem of Larson and Radford \cite{MR957441}, $H$ is involutory. Let $\Vect_{\fdim}(k)$ denote the category of finite-dimensional vector spaces over $k$. The forgetful functor $\fdMod(H) \to \Vect_\fdim(k)$ is a separable functor with $k$-linear section
\begin{gather*}
  \Pi_{V,W}: \Hom_k(V, W) \to \Hom_H(V, W)
  \quad (V, W \in \Mod(H)) \\
  \Pi_{V,W}(f)(v) = S(\Lambda_{(1)}) f(\Lambda_{(2)} v)
  \quad (f \in \Hom_k(V, W), v \in V)
\end{gather*}
where $\Lambda \in H$ is the {\em Haar integral} ({\em i.e.}, the two-sided integral such that $\varepsilon(\Lambda) = 1$). Let $\Bil(V)$ denote the set of all bilinear forms on $V$. Instead of the map \eqref{eq:transposition-sep}, we prefer to consider the map
\begin{equation*}
  \widetilde{\Sigma}_V: \Bil(V) \to \Bil(V),
  \quad \widetilde{\Sigma}_V(b)(v, w) = b(\Lambda_{(1)} v, \Lambda_{(2)} w),
  \quad (b \in \Bil(V), v, w \in V)
\end{equation*}
which makes the diagram
\begin{equation*}
  \begin{CD}
    \Bil(V) & \ \cong \ & \Hom_k(V, V^\vee) @>{\Pi_{V,V^\vee}}>> \Hom_H(V, V^\vee) & \ \cong \ & \Bil_H(V) \\
    @V{\widetilde{\Sigma}_V}VV @V{\widetilde{\Trans}_{V,V}}VV @VV{\Trans_{V,V}}V @VV{\Sigma_V}V \\
    \Bil(V) & \ \cong \ & \Hom_k(V, V^\vee) @<{}<{\text{inclusion}}< \Hom_H(V, V^\vee) & \ \cong \ & \Bil_H(V)
  \end{CD}
\end{equation*}
commutes; see \S\ref{sec:piv-alg}, especially Theorem~\ref{thm:piv-FS-ind-ch}, for the details. Let $\rho: H \to \End_k(V)$ be the algebra map corresponding to the action $H \otimes V \to V$. By Lemma~\ref{lem:trace-1}, we have
\begin{equation*}
  \nu(V) = \Trace(\widetilde{\Sigma}_{V})
  = \Trace \Big( \rho(\Lambda_{(1)}) \circ \rho(\Lambda_{(2)}) \Big)
  = \chi_V(\Lambda_{(1)} \Lambda_{(2)})
\end{equation*}
where $\chi_V = \Trace \circ \rho$. This is the FS indicator for $H$ introduced by Linchenko and Montgomery \cite{MR1808131}. In the case where $H = k G$, we have $\nu(V) = |G|^{-1} \sum_{g \in G} \chi_V(g^2)$ since $\Lambda = |G|^{-1} \sum_{g \in G} g$. Formula~\eqref{eq:FS-formula-cpt} for compact groups is obtained in a similar way; see \S\ref{sec:coalgebras} for details.

\subsection{Twisted Duality}
\label{subsec:tw-FS-ind}

To deal with some twisted versions of the Frobenius--Schur theorem, we introduce a {\em twisting adjunction} of a category with duality and give a method to twist the original duality functor by using a twisting adjunction. Our method can be thought as a generalization of the arguments in \cite[\S4]{MR2879228}.

Let $\mathcal{C}$ be a category with duality. Suppose that we are given an adjunction $(F, G, \eta, \varepsilon): \mathcal{C} \to \mathcal{C}$ and a natural transformation $\xi_X: F(X^\vee) \to G(X)^\vee$. Define $\zeta_X$ by
\begin{equation}
  \label{eq:duality-trans-1}
  \begin{CD}
    \zeta_X: G(X^\vee) @>{j_{G(X^\vee)}}>> G(X^\vee)^{\vee\vee}
    @>{\xi^\vee_{X^\vee}}>> F(X^{\vee\vee})^\vee
    @>{F(j_X)^\vee}>> F(X)^\vee
  \end{CD}
\end{equation}
$\zeta_X$ is a natural transformation making the following diagrams commute:
\begin{equation}
  \label{eq:duality-trans-2}
  \begin{CD}
    F(X) @>{F(j_X)}>> F(X^{\vee\vee})\phantom{,} \\
    @V{j_{F(X)}}VV @VV{\xi_{X^\vee}}V \\
    F(X)^{\vee\vee} @>>{\zeta_X^\vee}> G(X^\vee)^\vee,
  \end{CD}
  \qquad \qquad
  \begin{CD}
    G(X) @>{G(j_X)}>> G(X^{\vee\vee})\phantom{.} \\
    @V{j_{G(X)}}VV @VV{\zeta_{X^\vee}}V \\
    G(X)^{\vee\vee} @>>{\xi_X^\vee}> F(X^\vee)^\vee
  \end{CD}
\end{equation}
Indeed, the commutativity of the first diagram is checked as follows:
\begin{align*}
  \zeta_X^\vee \circ j_{F(X)}
  = j_{G(X^\vee)}^\vee \circ \xi_{X^\vee}^{\vee\vee} \circ F(j_{X})^{\vee\vee} \circ j_{F(X)}
  = j_{G(X^\vee)}^\vee \circ j_{G(X^\vee)^\vee} \circ \xi_{X^\vee} \circ F(j_{X})
  = \xi_{X^\vee} \circ F(j_{X})
\end{align*}
The commutativity of the second one is checked in a similar way as follows:
\begin{align*}
  \zeta_{X^\vee} \circ G(j_X)
  & = F(j_{X^\vee})^\vee \circ \xi_{X^{\vee\vee}}^\vee \circ j_{G(X^{\vee\vee})} \circ G(j_X) \\
  & = F(j_{X^\vee})^\vee \circ \xi_{X^{\vee\vee}}^\vee \circ G(j_X)^{\vee\vee} \circ j_{G(X)}
  = F(j_{X^\vee})^\vee \circ F(j_X^\vee)^{\vee} \circ \xi_{X}^\vee \circ j_{G(X)}
  = \xi_{X}^\vee \circ j_{G(X)}
\end{align*}
Now we define the {\em twisted duality functor} by $X^\sharp = G(X^\vee)$. The problem is when $\mathcal{C}$ is a category with duality with this new duality functor $(-)^\sharp$.

\begin{lemma}
  \label{lem:twisting-adjunction}
  Define $\omega: \id_\mathcal{C} \to (-)^\sharp \circ (-)^\sharp$ by
  \begin{equation}
    \label{eq:twisted-piv}
    \begin{CD}
      \omega_X: X @>{\eta_X}>> GF(X) @>{GF(j)}>> GF(X^{\vee\vee})
      @>{G(\xi_{X^\vee})}>> G(G(X^\vee)^\vee) = X^{\sharp\sharp}
    \end{CD}
  \end{equation}
  The triple $(\mathcal{C}, (-)^\sharp, \omega)$ is a category with duality if
  \begin{equation}
    \label{eq:duality-trans-3}
    \begin{CD}
      \Big( FG(X^\vee)
      @>{F(\zeta_X)}>> F(F(X)^\vee)
      @>{\xi_{F(X)}}>> (GF(X))^\vee
      @>{\eta_X^\vee}>> X^\vee
      \Big) = \varepsilon_{X^\vee}
    \end{CD}
  \end{equation}
  holds for all $X \in \mathcal{C}$. If, moreover, $\mathcal{C}$ is strong, then the following are equivalent: \\
  \indent {\rm (1)} $F$ is an equivalence. \\
  \indent {\rm (2)} $G$ is an equivalence. \\
  \indent {\rm (3)} $(\mathcal{C}, (-)^\sharp, \omega)$ is strong.
\end{lemma}
\begin{proof}
  Let $X \in \mathcal{C}$. By~\eqref{eq:duality-trans-1}, $\omega_X$ can be expressed in two ways as follows:
  \begin{equation*}
    \omega_X = G(\xi_{X^\vee} \circ F(j_{X})) \circ \eta_X = G(\zeta_{X} \circ j_{F(X)}) \circ \eta_X
  \end{equation*}
  For each $X \in \mathcal{C}$, we compute
  \begin{align*}
    \omega_X^\sharp \circ \omega_{X^\sharp}
    & = G(\omega_X^\vee) \circ \omega_{G(X^\vee)}
    = G(\omega_X^\vee) \circ G\Big( \xi_{G(X^\vee)^\vee} \circ F(j_{G(X^\vee)}) \Big) \circ \eta_{G(X^\vee)} \\
    & = G \Big( \eta_X^\vee \circ G(\zeta_X^\vee \circ j_{F(X)})^\vee
    \circ \xi_{G(X^\vee)^\vee} \circ F(j_{G(X^\vee)}) \Big) \circ \eta_{G(X^\vee)} \\
    & = G \Big(\eta_X^\vee \circ \xi_{F(X)} \circ F(j_{F(X)}^\vee \circ j_{F(X)^\vee} \circ \zeta_X) \Big) \circ \eta_{G(X^\vee)} \\
    & = G \Big(\eta_X^\vee \circ \xi_{F(X)} \circ F(\zeta_X) \Big) \circ \eta_{G(X^\vee)}
    = G (\varepsilon_{X^\vee}) \circ \eta_{G(X^\vee)} = \id_{G(X^\vee)} = \id_{X^\sharp}
  \end{align*}
  Here, the fourth equality follows from the naturality of $\xi$, the fifth from~\eqref{eq:duality-1}, the sixth from the Assumption~\eqref{eq:duality-trans-3}, and the seventh from \eqref{eq:adjunction-2}. Now we have shown that the triple $(\mathcal{C}, (-)^\sharp, \omega_X)$ is a category with duality.

  It is easy to prove the rest of the statement; $(1)\Leftrightarrow(2)$ follows from basic properties of adjunctions. To show $(2)\Leftrightarrow(3)$, recall Lemma~\ref{lem:duality-strong}.
\end{proof}

In view of Lemma~\ref{lem:twisting-adjunction}, we call a quintuple $\mathbf{t} = (F, G, \eta, \varepsilon, \xi)$ satisfying \eqref{eq:duality-trans-3} a {\em twisting adjunction} for $\mathcal{C}$. Given such a quintuple $\mathbf{t}$, we denote by $\mathcal{C}^{\mathbf{t}}$ the triple $(\mathcal{C}, (-)^\sharp, \omega)$ constructed in Lemma \ref{lem:twisting-adjunction}. Now we introduce an involution of a category with duality:

\begin{definition}
  \label{def:duality-involution}
  An {\em involution} of $\mathcal{C}$ is a triple $\mathbf{t} = (F, \xi, \eta)$ such that $(F, \xi)$ is a strong duality preserving functor on $\mathcal{C}$ and $\eta$ is an isomorphism
  \begin{equation}
    \label{eq:cat-w-dual-invol-1}
    \eta: (\id_\mathcal{C}, \id_{(-)^\vee}) \to (F, \xi) \circ (F, \xi)
  \end{equation}
  of duality preserving functors satisfying
  \begin{equation}
    \label{eq:cat-w-dual-invol-2}
    \eta_{F(X)} = F(\eta_X)
  \end{equation}
  for all $X \in \mathcal{C}$. We say that $\mathbf{t}$ is {\em strict} if $\xi$ and $\eta$ are identities.
\end{definition}

Note that \eqref{eq:cat-w-dual-invol-1} is an isomorphism of such functors if and only if
\begin{equation}
  \label{eq:cat-w-dual-invol-3}
  \eta_X^\vee \circ \xi_{F(X)} \circ F(\xi_X) \circ \eta_{X^\vee} = \id_{X^\vee}
\end{equation}
holds for all $X \in \mathcal{C}$.

An involution $(F, \xi, \eta)$ of $\mathcal{C}$ is a special type of twisting adjunction; indeed, by~\eqref{eq:cat-w-dual-invol-2}, the quadruple $(F, F, \eta, \eta^{-1})$ is an adjunction. By the definition of duality preserving functors, the natural transformation \eqref{eq:duality-trans-1} is given by
\begin{align*}
  \zeta_X
  & = F(j_{X})^\vee \circ \xi_{X^{\vee}}^\vee \circ j_{F(X^{\vee})}
  = F(j_{X})^\vee \circ \xi_{X^{\vee\vee}} \circ F(j_{X^{\vee}})
  = \xi_{X^{\vee\vee}} \circ F(j_X^\vee) \circ F(j_{X^{\vee}}) = \xi_X
\end{align*}
Since the counit is $\eta^{-1}$, \eqref{eq:duality-trans-3} is equivalent to \eqref{eq:cat-w-dual-invol-3}. Hence $(F, F, \eta, \eta^{-1}, \xi)$ is a twisting adjunction for $\mathcal{C}$. From now, we identify an involution of $\mathcal{C}$ with the corresponding twisting adjunction.

Suppose that $\mathcal{C}$ is a category with duality over $k$. Let $\mathbf{t} = (F, G, \dotsc)$ be a twisting adjunction for $\mathcal{C}$. $\mathbf{t}$ is said to be {\em $k$-linear} if the functor $F: \mathcal{C} \to \mathcal{C}$ is $k$-linear. If this is the case, $G$ is also $k$-linear as the right adjoint of $F$  (see,  e.g., \cite[IV.1]{MR1712872}) and hence $\mathcal{C}^\mathbf{t}$ is a category with duality over $k$.

\begin{definition}
  \label{def:tw-FS-ind}
  Let $\mathcal{C}$ be a category with duality over $k$, and let $\mathbf{t}$ be a $k$-linear twisting adjunction for $\mathcal{C}$. The ($\mathbf{t}$-){\em twisted FS indicator} $\nu^\mathbf{t}(X)$ of $X \in \mathcal{C}$ is defined by
  \begin{equation*}
    \nu^\mathbf{t}(X) = \nu(X^\mathbf{t})
  \end{equation*}
  where $X^\mathbf{t} \in \mathcal{C}^{\mathbf{t}}$ is the object $X$ regarded as an object of $\mathcal{C}^\mathbf{t}$.
\end{definition}

To study the twisted FS indicator, it is useful to introduce the twisted transposition map. Let $\mathcal{C}$ be a category with duality (not necessarily over $k$). Given a twisting adjunction $\mathbf{t} = (F, G, \eta, \varepsilon, \xi)$ for $\mathcal{C}$, the ($\mathbf{t}$-){\em twisted transposition map}
\begin{equation}
  \label{eq:tw-trans-1}
  \Trans_{X,Y}^\mathbf{t}: \Hom_\mathcal{C}(F(X), Y^\vee) \to \Hom_\mathcal{C}(F(Y), X^\vee)
  \quad (X, Y \in \mathcal{C})
\end{equation}
is defined for $f: F(X) \to Y^\vee$ by
\begin{equation*}
  \begin{CD}
    \Trans_{X,Y}^\mathbf{t}(f): F(Y) @>{F(f^\vee \circ j_Y)}>> F(F(X)^\vee)
    @>{\xi_{F(X)}}>> (GF(X))^\vee @>{\eta_X^\vee}>> X^\vee
  \end{CD}
\end{equation*}
For a while, let $\Trans^\sharp_{X,Y}: \Hom_\mathcal{C}(X, Y^\sharp) \to \Hom_\mathcal{C}(Y, X^\sharp)$ denote the transposition map for $\mathcal{C}^\mathbf{t}$. One can easily check that the diagram
\begin{equation}
  \label{eq:tw-trans-2}
  \begin{CD}
    \Hom_\mathcal{C}(F(X), Y^\vee) @>{\cong}>> \Hom_\mathcal{C}(X, G(Y^\vee))
    & \ = \ & \Hom_\mathcal{C}(X, Y^\sharp) \\
    @V{\Trans_{X,Y}^\mathbf{t}}VV @. @VV{\Trans_{X,Y}^\sharp}V \\
    \Hom_\mathcal{C}(F(Y), X^\vee) @>>{\cong}> \Hom_\mathcal{C}(Y, G(X^\vee))
    & \ = \ & \Hom_\mathcal{C}(Y, X^\sharp)
  \end{CD}
\end{equation}
commutes for all $X, Y \in \mathcal{C}$, where the horizontal arrows are the natural bijection given by~\eqref{eq:adjunction-3}. Hence, under the condition of Definition~\ref{def:tw-FS-ind}, we have
\begin{equation*}
  \nu^\mathbf{t}(X) = \Trace(\Trans^\sharp_{X,X}) = \Trace(\Trans_{X,X}^\mathbf{t})
\end{equation*}

The properties of the twisted FS indicator can be obtained as follows: Apply previous results to $\mathcal{C}^\mathbf{t}$ and then interpret the results in terms of $\mathcal{C}$ and $\mathbf{t}$ by using the commutative Diagram~\eqref{eq:tw-trans-2}. Following this scheme, a twisted version of Proposition~\ref{prop:FS-ind-abs-simple} is established as follows:

\begin{proposition}
  \label{prop:tw-FS-ind-abs-simple}
  Let $\mathcal{C}$ be an Abelian category with strong duality over $k$, let $X \in \mathcal{C}$, and let $\mathbf{t} = (F, G, \dotsc)$ be a $k$-linear twisting adjunction for $\mathcal{C}$ such that $F$ is an equivalence of categories. Then:

  {\rm (a)} If $X$ is a finite biproduct of simple objects, then $\nu^\mathbf{t}(F(X)) = \nu^\mathbf{t}(X^{\vee})$.

  {\rm (b)} If $X$ is absolutely simple, then $\nu^\mathbf{t}(X) \in \{ 0, \pm 1 \}$. $\nu^\mathbf{t}(X) \ne 0$ if and only if $F(X) \cong X^\vee$.
\end{proposition}

Now let $H$ be a finite-dimensional semisimple Hopf algebra over an algebraically closed field $k$ of characteristic zero. We give two examples of $k$-linear twisting adjunctions for $\fdMod(H)$.

\begin{example}
  \label{ex:tw-adj-1}
  An automorphism $\tau$ of $H$ induces a strict monoidal autoequivalence on $\fdMod(H)$. If $\tau^2 = \id_H$, then it gives rise to a strict involution of $\fdMod(H)$, which we denote by the same symbol $\tau$. The study of the $\tau$-twisted FS indicator leads us to the results of Sage and Vega \cite{MR2879228}; see \S\ref{sec:piv-alg}.
\end{example}

\begin{example}
  Let $L \in \fdMod(H)$ be a left $H$-module such that $h_{(1)} \ell \otimes h_{(2)} = h_{(2)} \ell \otimes h_{(1)}$ holds for all $h \in H$ and $\ell \in L$. This condition implies that, for each $X \in \fdMod(H)$, the map
  \begin{equation*}
    \mathrm{flip}: L \otimes X \to X \otimes L,
    \quad \ell \otimes x \mapsto x \otimes \ell
    \quad (\ell \in L, x \in X)
  \end{equation*}
  is an isomorphism of left $H$-modules. Fix a basis $\{ \ell_i \}$ of $L$ and define
  \begin{gather*}
    \eta_X: X \to L \otimes L^\vee \otimes X,
    \quad \eta_X(x) = \sum \ell_i \otimes \ell_i^\vee \otimes x \\
    \varepsilon_X: L^\vee \otimes L \otimes X \to X,
    \quad \varepsilon_X(\ell_i^\vee \otimes \ell_j \otimes x) = \delta_{i j} x
  \end{gather*}
  for $x \in X$, where $\{ \ell_i^\vee \}$ is the dual basis of $\{ \ell_i \}$. The quadruple $(F = L^\vee \otimes (-), G = L \otimes (-), \eta, \varepsilon)$ is an adjunction on $\fdMod(H)$. Moreover, if we define $\xi_X$ by
  \begin{equation*}
    \begin{CD}
      \xi_X: L^\vee \otimes X^\vee @>{\cong}>> (X \otimes L)^\vee @>{\mathrm{flip}^\vee}>> (L \otimes X)^\vee
    \end{CD}
  \end{equation*}
  then the quintuple $\mathbf{t}(L):  = (F, G, \eta, \varepsilon, \xi)$ is a twisting adjunction for $\fdMod(H)$. Since, in general, $F$ and $G$ are not monoidal, $\mathbf{t}(L)$ is of different type of twisting adjunctions from Example~\ref{ex:tw-adj-1}.

  Following the above notation, we write $X^\sharp$ for $G(X^\vee)$. To interpret the $\mathbf{t}(L)$-twisted FS indicator $\nu(V; L) := \nu^{\mathbf{t}(L)}(V)$, we recall that there is an isomorphism
  \begin{equation*}
    \Hom_H(X, Y^\sharp) = \Hom_H(X, L \otimes Y^\vee) \cong \Hom_H(X \otimes Y, L)
  \end{equation*}
  natural in $X, Y \in \fdMod(H)$. The transposition map for $\fdMod(H)^{\mathbf{t}(L)}$ induces
  \begin{equation*}
    \Sigma_V^L: \Hom_H(V \otimes V, L) \to \Hom_H(V \otimes V, L),
    \quad \Sigma_V^L(b)(v, w) = b(w, v)
  \end{equation*}
  via the above isomorphism. Hence we have $\nu(V; L) = \dim_k B_H^+(V; L) - \dim_k B_H^-(V; L)$, where $B_H^{\pm}(V; L)$ is the eigenspace of $\Sigma_V^L$ with eigenvalue $\pm 1$.

  Let $B(V; L)$ denote the set of all bilinear maps $V \times V \to L$. To express $\nu(V; L)$ by using the characters of $V$ and $L$, we use the map
  \begin{equation*}
    \widetilde{\Sigma}_V^L: B(V; L) \to B(V; L),
    \quad \widetilde{\Sigma}_V^L(b)(v, w)
    = S(\Lambda_{(1)}) \cdot b(\Lambda_{(2)} w, \Lambda_{(3)} v)
  \end{equation*}
  which makes the following diagrams commute:
  \begin{equation*}
    \begin{CD}
      B(V; L) & \ \cong \ & \Hom_k(V \otimes V, L) @>{\Pi_{V \otimes V, L}}>>
      \Hom_H(V \otimes V, L) & \ \cong \ & \Hom_H(V, V^\sharp)\phantom{.} \\
      @V{\widetilde{\Sigma}_V^L}VV @. @VV{\Sigma_{V}^{L}}V @VV{\Trans_{V,V}^\sharp}V \\
      B(V; L) & \ \cong \ & \Hom_k(V \otimes V, L) @<{}<{\text{inclusion}}<
      \Hom_H(V \otimes V, L) & \ \cong \ & \Hom_H(V, V^\sharp)
    \end{CD}
  \end{equation*}
  Now, let, in general, $f: A \to B$, $g: B \to A$ and $h: M \to M$ be linear maps between finite-dimensional vector spaces. Then, in a similar way as Lemma~\ref{lem:trace-1}, one can show that the trace of
  \begin{equation*}
    T: \Hom_k(A \otimes B, M) \to \Hom_k(A \otimes B, M),
    \quad T(\mu)(a \otimes b) = h \mu(g(b) \otimes f(a))
  \end{equation*}
  is given by $\Trace(T) = \Trace(h) \Trace(f g)$. By using this formula, we have
  \begin{equation*}
    \nu(V; L) = \Trace(\widetilde{\Sigma}_{V}^{L}) = \chi_L(S(\Lambda_{(1)})) \chi_V(\Lambda_{(2)} \Lambda_{(3)})
  \end{equation*}
\end{example}

If $\dim_k L = 1$, then $L \otimes (-)$ is an equivalence and hence Proposition~\ref{prop:tw-FS-ind-abs-simple} can be applied to the above example. By the above arguments, we now obtain in the following another type of twisted version of the Frobenius--Schur theorem for semisimple Hopf algebras.

\begin{theorem}
  \label{thm:FS-ind-Hopf-tw-2}
  Let $\alpha: H \to k$ be an algebra map such that $\alpha(h_{(1)}) h_{(2)} = \alpha(h_{(2)}) h_{(1)}$ holds for all $h \in H$ (or, equivalently, let $\alpha$ be a central grouplike element of the dual Hopf algebra $H^\vee$), and let $L$ be the left $H$-module corresponding to $\alpha$. Then, for all simple module $V \in \fdMod(H)$, we have
  \begin{equation*}
    \nu(V; L) = \alpha(S(\Lambda_{(1)})) \chi_V(\Lambda_{(2)} \Lambda_{(3)}) \in \{ 0, \pm 1 \}
  \end{equation*}
  Moreover, for a simple module $V \in \fdMod(H)$, the following statements are equivalent: \\
  \indent {\rm (1)} $\nu(V; L) \ne 0$. \\
  \indent {\rm (2)} $V \cong L \otimes V^\vee$. \\
  \indent {\rm (3)} There exists a non-degenerate bilinear form $b: V \otimes V \to k$ satisfying
  \begin{equation*}
    b(h_{(1)} v, h_{(2)} w) = \alpha(h) b(v, w)
    \text{\quad for all $h \in H$ and $v, w \in V$}
  \end{equation*}
  If one of the above equivalent statements holds, then such a bilinear form $b$ is unique up to scalar multiples and satisfies $b(w, v) = \nu(V; L) b(v, w)$ for all $v, w \in V$.
\end{theorem}

For the case where $H = k G$ is the group algebra of a finite group $G$, the above theorem has been obtained by Mizukawa \cite[Theorem~3.5]{MR2787651}.

\subsection{Group Action on a Pivotal Monoidal Categories}
\label{subsec:group-action-pivotal}

We have concentrated on studying generalizations of the second FS indicator. Here we briefly explain how to define the higher twisted FS indicators for $k$-linear pivotal monoidal categories by generalizing those for semisimple Hopf algebras due to Sage and Vega \cite{MR2879228}. As we have remarked in Section 1, the details are left for future work.

For a set $S$, we denote by $\underline{S}$ the category whose objects are the elements of $S$ and whose morphisms are the identity morphisms.  If $G$ is a group, then $\underline{G}$ is a strict monoidal category with tensor product given by $x \otimes y = x y$ ($x, y \in G$).

Let $\mathcal{C}$ be a $k$-linear pivotal monoidal category with pivotal structure $j$. We denote by $\underline{\mathrm{Aut}}_\mathrm{piv}(\mathcal{C})$ the category of $k$-linear monoidal autoequivalences of $\mathcal{C}$ that {\em preserve the pivotal structure} in the sense of \cite{MR2381536}. This is a strict monoidal category with respect to the composition of monoidal functors. By an {\em action} of $G$ on $\mathcal{C}$, we mean a strong monoidal functor
\begin{equation*}
  \underline{G} \to \underline{\mathrm{Aut}}_{\mathrm{piv}}(\mathcal{C}), \quad g \mapsto F_g \quad (g \in G)
\end{equation*}
Note that, by definition, there are natural isomorphisms $\id_\mathcal{C} \cong F_1$ and $F_{x} \circ F_{y} \cong F_{x y}$ of monoidal functors. We say that an action $\underline{G} \to \underline{\mathrm{Aut}}_{\mathrm{piv}}(\mathcal{C})$ is {\em strict} if it is strict as a monoidal functor and, moreover, $F_g: \mathcal{C} \to \mathcal{C}$ is strict as a monoidal functor for all $g \in G$.

Now suppose that an action of $G$ on $\mathcal{C}$ is given. The {\em crossed product} $\mathcal{C} \rtimes G$ is a monoidal category defined as follows: As a $k$-linear category, $\mathcal{C} \rtimes G = \bigoplus_{g \in G} \mathcal{C} \rtimes g$, where $\mathcal{C} \rtimes g = \mathcal{C}$ is a copy of $\mathcal{C}$. Given an object $X \in \mathcal{C}$, we denote by $(X, g)$ the object $X$ regarded as an object of $\mathcal{C} \rtimes g \subset \mathcal{C} \rtimes G$. The tensor product of $\mathcal{C} \rtimes G$ is given by
\begin{equation*}
  (X, g) \otimes (Y, g') = (X \otimes F_g(Y), g g')
\end{equation*}
see \cite{MR1815142} and \cite{MR2480712} for details. We now claim:

\begin{lemma}
  \label{lem:crossed-pivotal}
  $\mathcal{C} \rtimes G$ is rigid. The dual object of $(X, g) \in \mathcal{C} \rtimes g$ is given by
  \begin{equation*}
    (X, g)^\vee = (F_{g^{-1}}(X^\vee), g^{-1})
  \end{equation*}
  Moreover, $\mathcal{C} \rtimes G$ is a pivotal monoidal category with pivotal structure given by
  \begin{equation*}
    \begin{CD}
      (X, g) = X @>{j_X}>> X^{\vee\vee} \cong F_g^{} F_{g^{-1}} (X^{\vee\vee})
      @>{F_g^{}(\xi_{g^{-1};X^{\vee}})}>>  F_{g}^{}(F_{g^{-1}}^{} (X^\vee)^\vee)
      = (X, g)^{\vee\vee}
    \end{CD}
  \end{equation*}
  where $\xi_{g^{-1};V}: F_{g^{-1}}(V^\vee) \to F_{g^{-1}}(V)^{\vee}$ is the duality transform \cite[\S1]{MR2381536} of $F_{g^{-1}}: \mathcal{C} \to \mathcal{C}$.
\end{lemma}

In the most important case for us where the action of $G$ is strict, this lemma is easy to prove. The proof for general cases is tedious and omitted for brevity.

\begin{definition}
  Let $\mathcal{C}$ be a $k$-linear pivotal monoidal category satisfying~\eqref{eq:assumption-fin-dim} and suppose that an action of $G$ on $\mathcal{C}$ is given. For a positive integer $n$ and an element $g \in G$, we define the {\em $g$-twisted $n$-th FS indicator $\nu_n^g(V)$} of $V \in \mathcal{C}$ by $\nu_n^g(V) = \nu_n((V, g))$, where $\nu_n$ in the right-hand side stands for the $n$-th FS indicator of Ng and Schauenburg \cite{MR2381536} for the $k$-linear pivotal monoidal category $\mathcal{C} \rtimes G$.  
\end{definition}

For a pair of positive integers $(n, r)$, the $(n, r)$-th FS indicator $\nu_{n,r}$ is also defined in \cite{MR2381536}. It is clear how to define the {\em $g$-twisted $(n, r)$-th FS indicator $\nu_{n,r}^g$} of $V \in \mathcal{C}$.

We explain that our definition agrees with that of Sage and Vega \cite{MR2879228}. For simplicity, we treat all monoidal categories as if they were strict. Note that, as an object of $\mathcal{C}$, we have
\begin{equation*}
  (V, g)^{\otimes n} = 
  V \otimes F_g^{}(V) \otimes F_g^2(V) \otimes \dotsb \otimes F_g^{n - 1} (V)
  \quad (:= \widetilde{V^{\otimes n}})
\end{equation*}
Now let $H$ be a finite-dimensional semisimple Hopf algebra over an algebraically closed field $k$ of characteristic zero. Then the group $G = \mathrm{Aut}_{\text{Hopf}}(H)^{\op}$ naturally acts on $\mathcal{C} = \fdMod(H)$. If $g: H \to H$ is an automorphism such that $g^n = \id_H$, then the map
\begin{equation*}
  E_{(V, g)}^{(n)}: \Hom_{\mathcal{C} \rtimes G}(\mathbf{1}_{\mathcal{C} \rtimes G}, (V, g)^{\otimes n})
  \to \Hom_{\mathcal{C} \rtimes G}(\mathbf{1}_{\mathcal{C} \rtimes G}, (V, g)^{\otimes n})
\end{equation*}
used to define the $n$-th FS indicator in \cite{MR2381536} coincides with the map
\begin{equation*}
  \alpha: (\widetilde{V^{\otimes n}})^H \to (\widetilde{V^{\otimes n}})^H, \quad
  \sum_{i_1, \dotsc, i_n} v_{i_1}^1 \otimes v_{i_2}^2 \dotsb \otimes v_{i_n}^n \mapsto
  \sum_{i_1, \dotsc, i_n} v_{i_1}^2 \otimes \dotsb \otimes v_{i_n}^n \otimes v_{i_1}^1
\end{equation*}
under the canonical identification $(\widetilde{V^{\otimes n}})^H = \Hom_{\mathcal{C} \rtimes G}(\mathbf{1}_{\mathcal{C} \rtimes G}, (V, g)^{\otimes n})$. Since the twisted FS indicator of \cite{MR2879228} is equal to the trace of the map $\alpha$ \cite[Theorem~3.5]{MR2879228}, our definition agrees with that \cite{MR2879228} in the case where both are defined.

Recall that a pivotal monoidal category is a category with duality. If $G = \langle a \mid a^2 = 1 \rangle$ acts on a $k$-linear pivotal monoidal category $\mathcal{C}$, then the functor $F_a: \mathcal{C} \to \mathcal{C}$ is naturally an involution of $\mathcal{C}$ in the sense of Definition~\ref{def:duality-involution}. Let $\mathbf{t}$ denote this involution. By using the crossed product $\mathcal{C} \rtimes G$, the category $\mathcal{C}^\mathbf{t}$ constructed in Lemma~\ref{lem:twisting-adjunction} can be described as follows:

\begin{proposition}
  $\mathcal{C}^{\mathbf{t}} \to \mathcal{C} \rtimes G$, $X \mapsto (X, a)$ is a $k$-linear fully faithful duality preserving functor.
\end{proposition}

This is clear from Lemma~\ref{lem:crossed-pivotal} and the definition of $\mathcal{C}^\mathbf{t}$.

\section{Pivotal Algebras}
\label{sec:piv-alg}

\subsection{Pivotal Algebras}

In this section, we introduce and study a class of algebras such that its representation category is a category with strong duality. We first recall that $\fdMod(H)$ is a pivotal monoidal category if $H$ is a pivotal Hopf algebra \cite{KMN09}. Note that the monoidal structure of $\fdMod(H)$ is defined by using the comultiplication of $H$. Since we do not need a monoidal structure, it seems to be a good way to consider ``pivotal Hopf algebras with no comultiplication''. This is the notion of {\em pivotal algebras}, which is formally defined as follows:

\begin{definition}
  \label{def:piv-alg}
  A {\em pivotal algebra} is a triple $(A, S, g)$ consisting of an algebra $A$, an anti-algebra map $S: A \to A$, and an invertible element $g \in A$ satisfying
  $S(g) = g^{-1}$ and $S^2(a) = g a g^{-1}$ for all $a \in A$.
\end{definition}

Let $A = (A, S, g)$ be a pivotal algebra. We denote by $\Mod(A)$ the category of left $A$-modules and by $\fdMod(A)$ its full subcategory of finite-dimensional modules. Given $V \in \Mod(A)$, we can make its dual space $V^\vee$ into a left $A$-module by
\begin{equation}
  \label{eq:dual-module}
  \langle a f, v \rangle := \langle f, S(a)v \rangle \quad (a \in A, f \in V^\vee, v \in V)
\end{equation}
The assignment $V \mapsto V^\vee$ extends to a contravariant endofunctor on $\Mod(A)$. Now, for  each $V \in \Mod(A)$, we define $j_V: V \to V^{\vee\vee}$ by
\begin{equation}
  \label{eq:piv-mor-alg}
  \langle j_V(v), f \rangle = \langle f, g v \rangle
  \quad (v \in V, f \in V^\vee)
\end{equation}
The following computation shows that $j_V: V \to V^{\vee\vee}$ is $A$-linear:
\begin{gather*}
  \langle a \cdot j_V(v), f \rangle
  = \langle S(a) f, g v \rangle
  = \langle f, S^2(a) g v \rangle
  = \langle f, g a v \rangle
  = \langle j_V(a v), f \rangle
\end{gather*}
It is obvious that $j_V$ is natural in $V \in \Mod(A)$. Now we verify~\eqref{eq:duality-1} as follows:
\begin{equation*}
  \langle (j_V^{})^\vee j_{V^\vee}^{}(f), v \rangle
  = \langle j_{V^\vee}^{}(f), j_V^{}(v) \rangle
  = \langle g f, g v \rangle
  = \langle f, S(g) g v \rangle
  = \langle f, v \rangle
\end{equation*}
Note that $j_V$ is an isomorphism if and only if $\dim_k V < \infty$. We conclude:

\begin{proposition}
  \label{sec:piv-alg-cat-dual}
  Let $A$ be a pivotal algebra. Then $\Mod(A) = (\Mod(A), (-)^\vee, j)$ is an Abelian category with duality over $k$. The full subcategory $\fdMod(A)$ is an Abelian category with strong duality over $k$.
\end{proposition}

Let $A$ and $B$ be algebras. Given an algebra map $f: A \to B$, we can make each left $B$-module into a left $A$-module by defining $a \cdot v = f(a) v$ ($a \in A, v \in V$). We denote by $f^\natural(V)$ the left $A$-module obtained in this way from $V$. The assignment $V \mapsto f^\natural(V)$ extends to a functor
\begin{equation}
  \label{eq:pull-back-func}
  f^\natural: \Mod(B) \to \Mod(A)
\end{equation}
By restriction, we also obtain a functor
\begin{equation}
  \label{eq:pull-back-func-fd}
  f^\natural|_\fdim: \fdMod(B) \to \fdMod(A)
\end{equation}
It is easy to see that these functors are $k$-linear, exact and faithful. If, moreover, $f$ is surjective, then they are full.

Suppose that $A = (A, S, g)$ and $B = (B, S', g')$ are pivotal algebras. A {\em morphism of pivotal algebras} from $A$ to $B$ is an algebra map $f: A \to B$ satisfying $f(g) = g'$ and $S'(f(a)) = f(S(a))$ for all $a \in A$. If $f$ is such a morphism, then the Functors \eqref{eq:pull-back-func} and \eqref{eq:pull-back-func-fd} are strict duality preserving functors.

An {\em involution} of $A$ is a morphism $\tau: A \to A$ of pivotal algebras such that $\tau^2 = \id_A$. Such a $\tau$ gives rise to a strict involution of $\Mod(A)$, which is usually denoted by the same symbol $\tau$. The proof of the following proposition is straightforward and omitted.

\begin{proposition}
  Let $A = (A, S, g)$ be a pivotal algebra, and let $\tau$ be an involution of $A$. Put  $S^\tau = S \circ \tau \, (=\tau \circ S)$. Then: \\
  \indent {\rm (a)} The triple $A^\tau = (A, S^\tau, g)$ is a pivotal algebra. \\
  \indent {\rm (b)} $\id_{\Mod(A)}$ is a strict duality preserving functor $\Mod(A^\tau) \to \Mod(A)^\tau$.
\end{proposition}

This implies that the $\tau$-twisted FS indicator of $V \in \fdMod(A)$ is equal to the untwisted FS indicator of $V$ regarded as a left $A^\tau$-module. Thus, in principle, the theory of the $\tau$-twisted FS indicator reduces to that of the untwisted FS indicator of $A^\tau$, which is again a pivotal algebra.

\subsection{FS Indicator for Pivotal Algebras}

Let $A = (A, S, g)$ be a pivotal algebra, and let $V \in \fdMod(A)$. Since $\fdMod(A)$ is a category with duality over $k$ satisfying \eqref{eq:assumption-fin-dim}, we can define $\nu(V)$ in the way of Section \ref{sec:categ-with-dual}.

The FS indicator $\nu(V)$ is interpreted as follows: Let $\Bil(V)$ be the set of all bilinear forms on $V$. Recall that there is a canonical isomorphism
\begin{equation}
  \label{eq:canonical-map}
  B_V: \Hom_k(V, V^\vee) \to \Bil(V), \quad
  B_V(f)(v, w) = \langle f(v), w \rangle
\end{equation}
Let $\Bil_A(V)$ be the subset of $\Bil(V)$ consisting of those $b \in \Bil(V)$ such that
\begin{equation}
  \label{eq:S-adjoint-bilin}
  b(a v, w) = b(v, S(a)w)
  \quad (a \in A, v, w \in V)
\end{equation}
The set $\Bil_A(V)$ is in fact the image of $\Hom_A(V,V^\vee) \subset \Hom_k(V, V^\vee)$ under the canonical isomorphism \eqref{eq:canonical-map}. Now we define $\Sigma_V: \Bil_A(V) \to \Bil_A(V)$ so that
\begin{equation*}
  \begin{CD}
    \Hom_A(V, V^\vee) @>{B_V}>> \Bil_{A}(V) \\
    @V{\Trans_{V,V}}VV @VV{\Sigma_V}V \\
    \Hom_A(V, V^\vee) @>>{B_V}> \Bil_{A}(V)
  \end{CD}
\end{equation*}
commutes. If $b = B_V(f)$ for some $f \in \Hom_A(V, V^\vee)$, then we have
\begin{equation}
  \label{eq:transposition-bilin-1}
  \begin{aligned}
    \Sigma_V(b)(v, w)
    = B_V(f^\vee j_V)(v, w)
    = \langle f^\vee j_V(v), w \rangle
    = \langle f(w), g v \rangle
    = b(w, g v)
  \end{aligned}
\end{equation}
for all $v, w \in V$. In view of \eqref{eq:transposition-bilin-1}, we set
\begin{equation*}
  \Bil_A^{\pm}(V) = \{ b \in \Bil_A(V) \mid \text{$b(w, g v) = \pm b(v, w)$ for all $v, w \in V$} \}
\end{equation*}
Then, as a counterpart of Proposition~\ref{prop:FS-ind-basic} (b), we have a formula
\begin{equation}
  \label{eq:FS-ind-bilin}
  \nu(V) = \dim_k \Bil_A^+(V) - \dim_k \Bil_A^-(V)
\end{equation}
Rephrasing the results of Section 2 by using these notations, we immediately obtain the following theorem:

\begin{theorem}
  \label{thm:piv-FS}
  If $V \in \fdMod(A)$ is absolutely simple, then we have $\nu(V) \in \{ 0, \pm 1 \}$.
  Moreover, the following are equivalent: \\
  \indent {\rm (1)} $\nu(V) \ne 0$. \\
  \indent {\rm (2)} $V$ is isomorphic to $V^\vee$ as a left $A$-module. \\
  \indent {\rm (3)} There exists a non-degenerate bilinear form $b$ on $V$ satisfying~\eqref{eq:S-adjoint-bilin}. \\
  If one of the above statements holds, then such a bilinear form $b$ is unique up to scalar multiples and satisfies $b(w, g v) = \nu(V) \cdot b(v, w)$ for all $v, w \in V$.
\end{theorem}

We denote by $\reg_A$ the left regular representation of $A$. If $A$ is finite-dimensional, then the FS indicator of $\reg_A$ is defined. In the case where $A = k G$ is the group algebra of a finite group $G$, there is a well-known formula $\nu(\reg_{k G}) = \# \{ x \in G \mid x^2 = 1 \}$. This formula is generalized to finite-dimensional pivotal algebras as follows:

\begin{theorem}
  \label{thm:piv-Tr-S}
  Let $A = (A, S, g)$ be a finite-dimensional pivotal algebra. \\
  \indent {\rm (a)} $\nu(\reg_A) = \Trace(Q)$, where $Q: A \to A$, $Q(a) = S(a)g$. \\
  \indent {\rm (b)} If $A$ is Frobenius, then $\reg_A \cong \reg_A^\vee$ as left $A$-modules. Hence, $\nu(\reg_A) = \nu(\reg_A^\vee)$.
\end{theorem}

The part (b) is motivated by Remark~\ref{rem:FS-higher}; as we have mentioned, our definition of the FS indicator is different from that of \cite{MR2381536}. Therefore, if, for example, $A$ is a finite-dimensional pivotal Hopf algebra, then there are two definitions of the FS indicator of the regular representation of $A$. Nevertheless they are equal since a finite-dimensional Hopf algebra is Frobenius.

\begin{proof}
  (a) Recall that there is an isomorphism $\Phi: \Hom_A(\reg_A, \reg_A^\vee) \to \reg_A^\vee$ given by $\Phi(f) = f(1)$. For $f \in \Hom_A(\reg_A, \reg_A^\vee)$ and $a \in A$, we have
  \begin{equation*}
    \langle \Phi \Trans_{A,A}(f), a \rangle
    = \langle (f^\vee j_V)(1), a \rangle
    = \langle j_V(1), f(a) \rangle
    = \langle f(a), g \rangle
  \end{equation*}
  Recalling that $f: \reg_A \to \reg_A^\vee$ is $A$-linear, we compute
  \begin{equation*}
    \langle f(a), g \rangle
    = \langle f(a \cdot 1), g \rangle
    = \langle a \cdot f(1), g \rangle
    = \langle \Phi(f), S(a) g \rangle,
    = \langle Q^\vee \Phi(f), a \rangle
  \end{equation*}
  Hence we have $\Trans_{A,A} = \Phi^{-1} \circ Q^\vee \circ \Phi^{}$ and therefore
  \begin{equation*}
    \nu(A) = \Trace(\Trans_{A,A}) = \Trace(Q^\vee) = \Trace(Q)
  \end{equation*}

  (b) Suppose that $A$ is Frobenius. By definition, there exists $\phi \in A^\vee$ such that the bilinear map $A \times A \to k$, $(a, b)\mapsto \phi(a b)$ ($a, b \in A$) is non-degenerate. By using $\phi$, we define a linear map $f: \reg_A \to \reg_A^\vee$ by $\langle f(a), b \rangle = \phi(S(a) b)$ ($a, b \in A$). It is obvious that $f$ is bijective. For $a, b, c \in A$, we compute
  \begin{equation*}
    \langle f(a b), c \rangle
    = \phi(S(a b) c)
    = \phi(S(b) S(a) c)
    = \langle f(b), S(a) c \rangle
    = \langle a \cdot f(b), c \rangle
  \end{equation*}
  Thus $f$ is $A$-linear and therefore $\reg_A \cong \reg_A^\vee$ as left $A$-modules.
\end{proof}

The following Theorem~\ref{thm:piv-Tr-Sv} is motivated by the {\em trace-like invariant} of Hopf algebras studied in \cite{MR2724230} and \cite{MR2804686}. Given $V \in \Mod(A)$, we denote by $\rho_V: A \to \End_k(V)$ the algebra map induced by the action of $A$. Let $I_V := \Ker(\rho_V)$ denote the annihilator of $V$. By the definition of the dual module $V^\vee$, we have $I_{V^\vee} = S(I_V)$. Hence, if $V$ is self-dual, then $\mathrm{Im}(\rho_V) \to \mathrm{Im}(\rho_V)$, $\rho_V(a) \mapsto \rho_V(S(a))$ ($a \in A$) is well-defined. We also note that if $V \in \fdMod(A)$ is absolutely simple, then:
\begin{equation}
  \label{eq:abs-simp-surj}
  \text{The algebra map $\rho_V: A \to \End_k(V)$ is surjective}
\end{equation}

\begin{theorem}
  \label{thm:piv-Tr-Sv}
  Let $V \in \fdMod(A)$ be an absolutely simple module. If $V$ is self-dual, then, by the above arguments, the map
  \begin{equation*}
    S_V: \End_k(V) \to \End_k(V), \quad \rho_V(a) \mapsto \rho_V(S(a)) \quad (a \in A)
  \end{equation*}
  is well-defined. By using $S_V$, we also define
  \begin{equation*}
    Q_V: \End_k(V) \to \End_k(V), \quad Q_V(f) = S_V(f) \circ \rho_V(g) \quad (f \in \End_k(V))
  \end{equation*}
  Then we have:
  \begin{equation*}
    {\rm (a)} \Trace(S_V) = \nu(V) \cdot \chi_V^{}(g), \qquad \qquad
    {\rm (b)} \Trace(Q_V) = \nu(V) \cdot \dim_k(V)
  \end{equation*}
\end{theorem}
\begin{proof}
  (a) This can be proved in the same way as \cite[Proposition 4.5]{KMN09}. Here we give another proof: Fix an isomorphism $p: V \to V^\vee$ of left $A$-modules and define $q: V \to V^\vee$ by $q = p^\vee \iota$, where $\iota = \iota_V$ is the canonical isomorphism \eqref{eq:k-Vec-piv-mor}. Our first claim is
  \begin{equation*}
    S_V(f) = q^{-1} f^\vee q \quad (f \in \End_k(V))
  \end{equation*}
  Indeed, \eqref{eq:abs-simp-surj}, there exists $a\ in A$ such that $f = \rho_V(a)$ for some $a \in A$. Hence, we compute
  \begin{align*}
    f^\vee q = (p \rho_V(a))^\vee \iota
    = (\rho_V(S a)^\vee p)^\vee \iota
    = p^\vee \rho_V(S a)^{\vee\vee} \iota
    = q \rho_V(S a) = q S_V(f)
  \end{align*}
  Next, we determine the map $V \otimes V^\vee \to V \otimes V^\vee$ induced by $S_V$ via
  \begin{equation*}
    V \otimes V^\vee \to \End_k(V), \quad v \otimes \lambda \mapsto (x \mapsto \lambda(x) v)
    \quad (\lambda \in V^\vee, v, x \in V)
  \end{equation*}
  If $f \in \End_k(V)$ is the element corresponding to $v \otimes \lambda \in V \otimes V^\vee$, then we have
  \begin{equation*}
    \langle f^\vee q(x), y \rangle
    = \langle p^\vee \iota(x), f(y) \rangle
    = \langle p(v), x \rangle \, \lambda(y)
    \quad (x, y \in V)
  \end{equation*}
  and therefore $S_V(f)(x) = \langle p(v), x \rangle q^{-1}(\lambda)$. This means that $S_V(f)$ corresponds to the element $q^{-1}(\lambda) \otimes p(v) \in V \otimes V^\vee$ via the above isomorphism.

  By the above observation, we have that the trace of $S_V$ is equal to that of
  \begin{equation*}
    V \otimes V^\vee \to V \otimes V^\vee,
    \quad v \otimes \lambda \mapsto q^{-1}(\lambda) \otimes p(v)
    \quad (v \in V, \lambda \in V^\vee)
  \end{equation*}
  Applying Lemma~\ref{lem:trace-1}, we have $\Trace(S_V) = \Trace(q^{-1} p)$. Now we recall the definition of the transposition map and compute $q = p^\vee \iota = p^\vee j_V \rho_V(g)^{-1} = \Trans_{V,V}(p) \rho_V(g)^{-1} = \nu(V) \cdot p \, \rho_V(g)^{-1}$. Hence, we conclude $\Trace(S_V) = \Trace(q^{-1}p) = \nu(V) \Trace(\rho_V(g)) = \nu(V) \chi_V(g)$.

  (b) The triple $E = (\End_k(V), S_V, \rho_V(g))$ is a pivotal algebra and $\rho_V: A \to E$ is a morphism of pivotal algebras. Let $V_0$ denote the vector space $V$ regarded as a left $E$-module. By Proposition~\ref{prop:FS-ind-inv}, the functor $\rho_V^\natural: \Mod(E) \to \Mod(A)$ preserves the FS indicator. Since $V = \rho_V^\natural(V_0)$, we have $\nu(V_0) = \nu(V)$.

  Now let $d = \dim_k(V)$. Then we have $\reg_E \cong V_0^{\oplus d}$ as left $E$-modules and therefore $\nu(E) = \nu(V_0) d = \nu(V) d$ by Proposition~\ref{prop:FS-ind-basic}. On the other hand, $\nu(E) = \Trace(Q_V)$ by Proposition~\ref{thm:piv-Tr-S}. Thus $\Trace(Q_V) = \nu(V) d$ follows.
\end{proof}

The following is a generalization of \cite[Theorem~8.8 (iii)]{MR2104908}.

\begin{corollary}
  \label{cor:Tr-S}
  Suppose that $k$ is algebraically closed and that $A = (A, S, g)$ is a finite-dimensional semisimple pivotal algebra. Let $\{ V_i \}_{i = 1, \dotsc, n}$ be a complete set of representatives of the isomorphism classes of simple left $A$-modules. Then
  \begin{equation*}
    \Trace(S) = \sum_{i = 1}^n \nu(V_i) \chi_{i}(g)
  \end{equation*}
  where $\chi_i = \chi_{V_i}^{}$ is the character of $V_i$.
\end{corollary}
\begin{proof}
  Put $I = \{ 1, \dotsc, n \}$. For $i \in I$, let $\rho_i: A \to \End_k(V_i)$ denote the action of $A$ on $V_i$. By the Artin--Wedderburn theorem, we have an isomorphism
  \begin{equation*}
    A \to \End_k(V_1) \oplus \dotsb \oplus \End_k(V_n),
    \quad a \mapsto (\rho_1(a), \dotsc, \rho_n(a))
  \end{equation*}
  of algebras. $S: A \to A$ induces an anti-algebra map
  \begin{equation*}
    \tilde{S}: \End_k(V_1) \oplus \dotsb \oplus \End_k(V_n) \to \End_k(V_1) \oplus \dotsb \oplus \End_k(V_n)
  \end{equation*}
  via the isomorphism. For each $i \in I$, we have $\tilde{S}(\End_k(V_i)) \subset \End_k(V_{i^*})$, where $i^* \in I$ is the element such that $V_i^\vee \cong V_{i^*}^{}$. Hence we obtain
  \begin{equation*}
    \Trace(S) = \Trace(\tilde{S})
    = \sum_{i \in I, i^* = i} \Trace(\widetilde{S}|_{\End_k(V_i)})
    = \sum_{i \in I, i^* = i} \nu(V_i) \chi_i(g)
  \end{equation*}
  by Theorem~\ref{thm:piv-Tr-Sv}. The sum in the right-hand side is equal to $\sum_{i = 1}^n \nu(V_i) \chi_i(g)$ since $\nu(V_i) = 0$ unless $i = i^*$. The proof is done.
\end{proof}

\subsection{Separable Pivotal Algebras}

Recall that an algebra $A$ is said to be {\em separable} if it has a {\em separability idempotent}, {\em i.e.}, an element $E \in A \otimes A$ such that $E^{1} E^{2} = 1$ and $a E^{1} \otimes E^{2} = E^{1} \otimes E^{2} a$ for all $a \in A$. If such an element exists, then the forgetful functor $\Mod(A) \to \Vect_\fdim$ is separable with section $\Pi_{V,W}: \Hom_k(V, W) \to \Hom_A(V, W)$ ($V, W \in \Mod(A)$) given by
\begin{equation*}
  \Pi_{V,W}(f)(v) = E^1 f(E^2 v)
  \quad (f \in \Hom_k(V, W))
\end{equation*}
Hence, if a pivotal algebra $A = (A, S, g)$ is separable (as an algebra), then we can apply the arguments of \S\ref{subsec:sep-functor}. This is a rationale for the following theorem:

\begin{theorem}
  \label{thm:piv-FS-ind-ch}
  Let $A = (A, S, g)$ be a separable pivotal algebra with separability idempotent $E \in A \otimes A$. Then, for all $V \in \fdMod(A)$, we have
  \begin{equation*}
    \nu(V) = \chi_V(S(E^1) g E^2)
  \end{equation*}
\end{theorem}
\begin{proof}
  Define $\widetilde{\Sigma}_V: \Bil(V) \to \Bil(V)$ so that the following diagrams commute:
  \begin{equation*}
    \begin{CD}
      \Bil(V) @>{B_V^{-1}}>> \Hom_k(V, V^\vee) @>{\Pi_{V,V^\vee}}>> \Hom_A(V, V^\vee) @>{B_V^{}}>> \Bil_A(V)\phantom{.} \\
      @V{\widetilde{\Sigma}_V}VV @V{\widetilde{\Trans}_{V,V}}VV @VV{\Trans_{V,V}}V @VV{\Sigma_V}V \\
      \Bil(V) @<<{B_V^{}}< \Hom_k(V, V^\vee) @<<{\text{inclusion}}< \Hom_A(V, V^\vee) @<<{B_V^{-1}}< \Bil_A(V)
    \end{CD}
  \end{equation*}
  By the arguments in \S\ref{subsec:sep-functor}, $\nu(V)$ is equal to $\Trace(\widetilde{\Trans}_{V,V})$. However, to make the computation easier, we prefer to compute $\Trace(\widetilde{\Sigma}_V)$, which is also equal to $\nu(V)$.

  Let $\Pi'_V: \Bil(V) \to \Bil_A(V)$ be the composition of the arrows of the first row of the above diagram. If $b = B_V(f)$ for some $f \in \Hom_k(V, V^\vee)$, we have
  \begin{gather*}
    \Pi'_V(b)(v, w)
    = \Big \langle \Pi_{V,V^\vee}(f)(v), w \Big \rangle
    = \Big \langle f (E^2v), S(E^1)w \Big \rangle
    = b \Big( E^2v, S(E^1)w \Big)
  \end{gather*}
  Hence, by~\eqref{eq:transposition-bilin-1}, we have
  \begin{equation*}
    \widetilde{\Sigma}_V(b)(v, w)
    = \Sigma_V \Big( \Pi'_V(b) \Big)(v, w)
    = b \Big( E^2 w, S(E^1) g v \Big)
  \end{equation*}
  Applying Lemma~\ref{lem:trace-1}, we obtain $\nu(V) = \chi_V(S(E^1) g E^2)$.
\end{proof}

We discuss the relation between Theorem~\ref{thm:piv-FS-ind-ch} and the results of \cite{MR1808131,MR2674691,2011arXiv1110.5672G}. Let $A = (A, S, g)$ be a pivotal algebra such that the algebra $A$ is symmetric with trace form $\phi: A \to k$. By definition, the map
\begin{equation}
  \label{eq:sym-alg-trace}
  A \times A \to k,
  \quad (a, b) \mapsto \phi(a b)
  \quad (a, b \in A)
\end{equation}
is a non-degenerate bilinear symmetric form. Fix a basis $\{ b_i \}_{i \in I}$ of $A$ and let $\{ b_i^\vee \}_{i \in I}$ be the dual basis of $\{ b_i \}$ with respect to~\eqref{eq:sym-alg-trace}. As remarked in \cite{MR2674691}, we have
\begin{equation*}
  \sum_{i \in I} a b_i \otimes b_i^\vee = \sum_{i \in I} b_i \otimes b_i^\vee a
\end{equation*}
for all $a \in A$. Hence $v_A = \sum_{i \in I} b_i b_i^\vee \in A$ is a central element, called the {\em volume} of $(A, \phi)$.

Now we suppose that the base field $k$ is of characteristic zero and $A$ is split semisimple over $k$. Then, as Doi showed in \cite{MR2674691}, the volume $v_A$ is invertible and hence $E = \sum_{i \in I} b_i \otimes b_i^\vee v_A^{-1} \in A \otimes A$ is a separability idempotent of $A$. By Theorem~\ref{thm:piv-FS-ind-ch}, the FS indicator of $V \in \Rep(A)$ is given by
\begin{equation*}
  \nu(V) = \sum_{i \in I} \chi_V(S(b_i) g b_i^\vee v_A^{-1})
\end{equation*}
If $V$ is simple, then, by Schur's lemma, $v_A$ acts on $V$ as $\chi_V(v_A) \dim_k(V)^{-1} \cdot \id_V$. Hence, we have
\begin{equation}
  \label{eq:FS-symm-alg-1}
  \nu(V) = \frac{\dim_k(V)}{\chi_V(v_A)} \sum_{i \in I} \chi_V(S(b_i) g b_i^\vee)
\end{equation}
The {\em Schur element} of a simple module $V \in \Rep(A)$ is given by $c_V = \chi_V(v_A) \dim_k(V)^{-2}$ (see Remark 1.6 of \cite{MR2674691}). By using the Schur element, $\nu(V)$ is expressed as
\begin{equation}
  \label{eq:FS-symm-alg-2}
  \nu(V) = \frac{1}{c_V \dim_k(V)} \sum_{i \in I} \chi_V(S(b_i) g b_i^\vee)
\end{equation}
Letting $g = 1$, we recover the results of \cite{MR1808131,MR2674691}. Geck \cite{2011arXiv1110.5672G} assumed that $g = 1$ and $A$ has a basis $\{ b_i \}_{i \in I}$ such that $b_i^\vee = S(b_i)$. If this is the case, then we have
\begin{equation*}
  \nu(V) = \frac{1}{c_V \dim_k(V)} \sum_{i \in I} \chi_V(b_i^2)
\end{equation*}

\begin{example}[Group-like algebras]
  \label{ex:FS-ind-ch-GLalg}
  As a generalization of the group algebra of a finite group and the adjacency algebra of an association scheme, Doi \cite{MR2051736,MR2674691} introduced a {\em group-like algebra}; it is defined to be a quadruple $(A, \varepsilon, \mathcal{B}, *)$ consisting of a finite-dimensional algebra $A$, an algebra map $\varepsilon: A \to k$, a basis $\mathcal{B} = \{ b_i \}_{i \in I}$ of $A$ indexed by a set $I$, and an involutive map $*: I \to I$, $i \mapsto i^*$ satisfying the following conditions: \\
  \indent (G0) There is a special element $0 \in I$ such that $b_0 = 1$ is the unit of $A$. \\
  \indent (G1) $\varepsilon(b_i) = \varepsilon(b_{i^*}) \ne 0$ for all $i \in I$. \\
  \indent (G2) $p_{i j}^k = p_{j^* i^*}^{k^*}$ for all $i, j, k \in I$, where $p_{i j}^k$ is given by $b_i \cdot b_j = \sum_{k \in I} p_{i j}^k b_k$ ($i, j, k \in I$). \\
  \indent (G3) $p_{i j}^0 = \delta_{i j^*} \varepsilon(b_i)$ for all $i \in I$. \\
  Now let $A = (A, \varepsilon, \mathcal{B}, *)$ be a group-like algebra. Define a linear map $S: A \to A$ by $S(b_i) = b_{i^*}$ for $i \in I$. One can check that the triple $A = (A, S, 1)$ is a pivotal algebra. By an {\em involution} of $A$, we mean an involutive map $\tau: I \to I$ satisfying
  \begin{equation*}
    \tau(i^*) = \tau(i)^* \text{\quad and \quad} p_{\tau(i),\tau(j)}^{\tau(k)} = p_{i j}^k
  \end{equation*}
  for all $i, j, k \in I$. Such a map gives rise to an involution of the pivotal algebra $(A, S, 1)$.

  In what follows, we compute the $\tau$-twisted FS indicator $\nu^\tau(V)$ of a simple module $V \in \Rep(A)$ under the assumption that the base field is $\mathbb{C}$ and $\varepsilon(b_i) > 0$ for all $i \in I$. Then, by the results of Doi \cite{MR2674691}, $A$ is semisimple. Note that $A$ is a symmetric algebra with trace form given by $\phi(b_i) = \delta_{i 0}$ and the dual basis of $\{ b_i \}$ with respect to \eqref{eq:sym-alg-trace} is given by $b_i^\vee = \varepsilon(b_i)^{-1} b_{i^*}$. Applying \eqref{eq:FS-symm-alg-1} and \eqref{eq:FS-symm-alg-2} to $A^\tau = (A, S \circ \tau, 1)$, we obtain the following formula:
  \begin{equation*}
    \nu^\tau(V)
    = \frac{\dim_\mathbb{C}(V)}{\chi_V(v_A)} \sum_{i \in I} \frac{1}{\varepsilon(b_i)} \chi_V^{}(b_{\tau(i)} b_i)
    = \frac{1}{c_V \dim_\mathbb{C}(V)} \sum_{i \in I} \frac{1}{\varepsilon(b_i)} \chi_V^{}(b_{\tau(i)} b_i)
  \end{equation*}
In particular, applying this formula to the adjacency algebra of an association scheme, we recover the formula of Hanaki and Terada \cite{TeradaJunya:2006-03}.

To obtain the twisted Frobenius--Schur theorem for this class of algebras, combine the above formula with Theorem~\ref{thm:piv-FS}; we then have $\nu^\tau(V) \in \{ 0, \pm 1 \}$. Moreover, $\nu^\tau(V) \ne 0$ if and only if there exists a non-degenerate bilinear form $\beta$ on $V$ such that $\beta(b_{\tau(i)} v, w) = \beta(v, b_{i^*} w)$ for all $i \in I$ and $v, w \in V$. Such a bilinear form $\beta$ is symmetric if $\nu^\tau(V) = +1$ and skew-symmetric if $\nu^\tau(V) = -1$.
\end{example}

\begin{example}[Weak Hopf $C^*$-algebras]
  \label{ex:FS-ind-ch-WHA}
We assume that the base field is $\mathbb{C}$. A {\em weak Hopf algebra} is an algebra $H$ which is a coalgebra at the same time such that there exists a special map $S: H \to H$ called the antipode; see \cite{MR1726707} and \cite{MR1793595} for the precise definition. We note that the antipode of a weak Hopf algebra is known to be an anti-algebra map.

  Let $H$ be a finite-dimensional weak Hopf $C^*$-algebra; see \cite[\S4]{MR1726707} for the precise definition. There exists an element $g \in H$, called the {\em canonical grouplike element} \cite[\S4]{MR1726707}, satisfying $S(g) = g^{-1}$, $S^2(x) = g x g^{-1}$ for all $x \in H$, and some other good properties. In particular, the triple $(H, S, g)$ is a pivotal algebra and therefore the FS indicator $\nu(V)$ is defined for each $V \in \fdMod(H)$.

  We can express $\nu(V)$ by using the {\em Haar integral} \cite[\S3]{MR1726707}; if $\Lambda \in H$ is the Haar integral in $H$, then $E = S(\Lambda_{(1)}) \otimes \Lambda_{(2)}$ is a separability idempotent of $H$ ({\em cf}. the proof of Theorem 3.13 of \cite{MR1726707}). Applying Theorem~\ref{thm:piv-FS-ind-ch}, we have
  \begin{equation}
    \label{eq:FS-formula-WHA}
    \nu(V) = \chi_V(S^2(\Lambda_{(1)}) g \Lambda_{(2)}) = \chi_V(g \Lambda_{(1)} \Lambda_{(2)})
  \end{equation}
  Combining the above formula with Theorem~\ref{thm:piv-FS}, we obtain the Frobenius--Schur theorem for semisimple weak Hopf algebras. We finally give some remarks concerning this example:

  (1) Takahiro Hayashi (in private communication with the author) has proved \eqref{eq:FS-formula-WHA} and analogous formulas of the higher FS indicators for weak Hopf algebras in the case where $S^2 = \id_H$.

  (2) The formula of Linchenko and Montgomery \cite{MR1808131} is the case where $H$ is an ordinary Hopf algebra. If this is the case, then the $C^*$-condition for $H$ is not needed since we have $S^2 = \id_H$ by the theorem of Larson and Radford \cite{MR957441}. It is not known whether every semisimple weak Hopf algebra $H$ has a grouplike element $g$ such that $S^2(h) = g h g^{-1}$ for all $h \in H$. An affirmative answer to this question proves Conjecture 2.8 of \cite{MR2183279}, which states that every fusion category admits a pivotal structure.

  (3) We do not know whether our Formula \eqref{eq:FS-formula-WHA} is equivalent to \cite[(4.3)]{MR1657800} or \cite[(3.70)]{MR1793595}. In \cite{MR2104908,MR1657800,MR1793595}, formulas are proved by finding a central element $e$ such that $\nu(V) = \chi_V(e)$ for all $V$. On the other hand, the element $S(E^1) g E^2$ of our Theorem~\ref{thm:piv-FS-ind-ch} is not central in general. In \S\ref{subsec:quasi-hopf-algebras}, we give a formula of the FS indicator for quasi-Hopf algebras and its twisted version. For the above reason, it is not straightforward to derive the formula of Mason and Ng \cite{MR2104908} from our formula.

  (4) By an involution of $H$, we mean an involutive algebra map $\tau: H \to H$ which is also a coalgebra map. Such a map $\tau$ is in fact an involution of the pivotal algebra $(H, S, g)$ and the $\tau$-twisted FS indicator of $V \in \Rep(H)$ is given by $\nu^\tau(V) = \chi_V(g \tau(\Lambda_{(1)}) \Lambda_{(2)})$. We omit the details since these results can be proved in a similar way as the case of quasi-Hopf algebras; see \S\ref{subsec:quasi-hopf-algebras}.
\end{example}

\begin{example}[Twisting by $L \otimes (-)$]
  By using separable pivotal algebras, Theorem~\ref{thm:FS-ind-Hopf-tw-2} can be proved as follows: Let $H$, $\alpha$ and $L$ be as in that theorem and define $T_\alpha: H \to H$ by $T_\alpha(h) = \alpha(h_{(1)}) S(h_{(2)})$ ($h \in H$). It is easy to see that the triple $H' = (H, T_\alpha, 1)$ is a pivotal algebra and the identity functor is a strict duality preserving functor between $\fdMod(H)^{\mathbf{t}(L)}$ and $\fdMod(H')$. Hence, by Proposition~\ref{prop:FS-ind-inv} and Theorem~\ref{thm:piv-FS-ind-ch}, we have
  \begin{equation*}
    \nu(V; L) = \chi_V(T_\alpha S(\Lambda_{(1)}) \Lambda_{(2)})
    = \alpha(S(\Lambda_{(1)})) \chi_V(\Lambda_{(2)} \Lambda_{(3)})
  \end{equation*}
  for all $V \in \fdMod(H)$. The meaning of $\nu(V; L)$ is obtained by applying Theorem~\ref{thm:piv-FS} to $H'$.
\end{example}

\subsection{Quasi-Hopf Algebras}
\label{subsec:quasi-hopf-algebras}

We derive a formula of Mason and Ng \cite{MR2104908} and its twisted version from our results.

Recall that a {\em quasi-Hopf algebra} \cite{MR1047964} is a data $H = (H, \Delta, \varepsilon, \Phi, S, \alpha, \beta)$ consisting of an algebra $H$, algebra maps $\Delta: H \to H \otimes H$ and $\varepsilon: H \to k$, an anti-algebra automorphism $S: H \to H$, elements $\alpha, \beta \in H$ and an invertible element $\Phi \in H^{\otimes 3}$ with inverse $\overline{\Phi}$ satisfying numerous conditions. Let $H_i = (H_i, \Delta_i, \varepsilon_i, \Phi_i, S_i, \alpha_i, \beta_i)$ be quasi-Hopf algebras $(i = 1, 2$). A {\em morphism of quasi-Hopf algebras} from $H_1$ to $H_2$ is an algebra map $f: H_1 \to H_2$ satisfying
\begin{gather*}
  \Delta_2 f = (f \otimes f)\Delta_1, \quad
  \varepsilon_2 f = \varepsilon_1, \quad
  \Phi_2 = (f \otimes f \otimes f)(\Phi_1) \\
  S_2 f = f S_1, \quad 
  \alpha_2 = f(\alpha_1), \quad 
  \beta_2 = f(\beta_2)
\end{gather*}
Hence, by an {\em involution} of a quasi-Hopf algebra $H$, we shall mean a morphism $\tau: H \to H$ of quasi-Hopf algebras such that $\tau^2 = \id_H$.

If $H$ is a quasi-Hopf algebra, then $\fdMod(H)$ is a rigid monoidal category. Given $V \in \fdMod(H)$, we denote by $e_V: V^\vee \otimes V \to k$ and $c_V: k \to V \otimes V^\vee$ the evaluation and the coevaluation, respectively. Here we need to recall that the dual module of $V$ is defined by the same way as~\eqref{eq:dual-module} and the maps $e_V$ and $c_V$ are given by
\begin{equation*}
  e_V(\lambda \otimes v) = \langle \lambda, \alpha v \rangle
  \quad (\lambda \in V^\vee, v \in V); \quad
  c_V(1) = \sum_{i = 1}^n \beta v_i \otimes v^i
\end{equation*}
where $\{ v_i \}_{i = 1, \dotsc, n}$ is a basis of $V$ and $\{ v^i \}$ is the dual basis.

We need additional assumptions on $H$ so that $H$ is a pivotal algebra. In what follows, we suppose that $k$ is an algebraically closed field of characteristic zero and $H$ is a finite-dimensional semisimple quasi-Hopf algebra. Then $\fdMod(H)$ is a fusion category \cite{MR2183279} such that each its object has an integral Frobenius--Perron dimension. Therefore, by the results of \cite{MR2183279}, $\fdMod(H)$ has a {\em canonical pivotal structure}, {\em i.e.}, an isomorphism $j: \id_{\fdMod(H)} \to (-)^{\vee\vee}$ of $k$-linear monoidal functors such that, for all $V \in \fdMod(H)$, the composition
\begin{equation*}
  \begin{CD}
    k @>{c_V}>> V \otimes V^\vee @>{j_V \otimes \id_{V^\vee}}>> V^{\vee\vee} \otimes V^\vee @>{e_{V^\vee}}>> k
  \end{CD}
\end{equation*}
maps $1 \in k$ to $\dim_k(V) \in k$. Now let $g \in H$ be the image of $1 \in H$ under
\begin{equation*}
  \begin{CD}
    H @>{j_H}>> H^{\vee\vee} @>{\iota_H^{-1}}>> H
  \end{CD}
\end{equation*}
where $\iota_H$ is the Isomorphism~\eqref{eq:k-Vec-piv-mor}. We call $g$ the {\em canonical pivotal element} of $H$. By definition, $g$ is invertible and satisfies $S(g) g = 1$ and $S^2(h) = g a g^{-1}$ for all $a \in H$; see \cite{MR2104908} and \cite{MR2095575} for details. Hence $(H, S, g)$ is a pivotal algebra. Now we remark:

\begin{lemma}
  \label{lem:can-piv-1}
  Let $f: H_1 \to H_2$ be an isomorphism between finite-dimensional semisimple quasi-Hopf algebras. Then we have $f(g_1) = g_2$, where $g_i \in H_i$ is the canonical pivotal element of $H_i$.
\end{lemma}
\begin{proof}
  $f$ induces a functor $f^\natural: \fdMod(H_1) \to \fdMod(H_2)$. By the definition of morphisms of quasi-Hopf algebras, the functor $f^\natural$ is a $k$-linear strict monoidal equivalence. The result follows from the fact that such a functor preserves the canonical pivotal structure \cite[Corollary~6.2]{MR2381536}.
\end{proof}

From this lemma, we see that an involution $\tau$ of the quasi-Hopf algebra $H$ is an involution of the pivotal algebra $(H, S, g)$. As we have observed in \S\ref{sec:piv-alg-cat-dual}, $\tau$ gives rise to an involution of $\fdMod(H)$ and hence the $\tau$-twisted FS indicator $\nu^\tau(V)$ is defined for $V \in \fdMod(H)$. By Theorem~\ref{thm:piv-FS} applied to $A = (H, S \circ \tau, g)$, we have the following property of $\nu^\tau$:

\begin{theorem}
  Let $V \in \fdMod(H)$ be a simple module. Then $\nu^\tau(V) \in \{ 0, \pm 1 \}$ and the following statements are equivalent: \\
  \indent {\rm (1)} $\nu^\tau(V) \ne 0$. \\
  \indent {\rm (2)} $\tau^\natural(V)$ is isomorphic to the dual module $V^\vee$ as a $H$-module. \\
  \indent {\rm (3)} There exists a non-degenerate bilinear form $b$ on $V$ satisfying \\
  \begin{equation*}
    b(\tau(h) v, w) = b(v, S(h) w) \text{\quad for all $v, w \in V$}
  \end{equation*}
  If one of the above statements holds, then such a bilinear form $b$ is unique up to scalar multiples and satisfies $b(w, g v) = \nu^\tau(V) \cdot b(v, w)$ for all $v, w \in V$.
\end{theorem}

Next we express the number $\nu^\tau(V)$ by using the character of $V$. To that end, it is sufficient to find a separability idempotent of $H$. Let $\Lambda \in H$ be the Haar integral of $H$ (see \cite{1999math4164H} and \cite{MR1615781}). We set
\begin{align*}
  p_L & = \Phi^2 S^{-1}(\Phi^1 \beta) \otimes \Phi^{3}, &
  q_L & = S(\overline{\Phi}{}^1) \alpha \overline{\Phi}{}^2 \otimes \overline{\Phi}{}^{3} \\
  p_R & = \overline{\Phi}{}^1 \otimes \overline{\Phi}{}^2 \beta S(\overline{\Phi}{}^1), &
  q_R & = \Phi^1 \otimes S^{-1}(\alpha \Phi^{3}) \Phi^2
\end{align*}
and fix $p \in \{ p_L, p_R \}$ and $q \in \{ q_L, q_R \}$. Following \cite[Lemma~3.1]{MR2104908}, we have
\begin{align}
  \label{eq:q-Hopf-integ-1}
  \Lambda_{(1)} p^1 a \otimes \Lambda_{(2)} p^2
  & = \Lambda_{(1)} p^1 \otimes \Lambda_{(2)} p^2 S(a) \\
  \label{eq:q-Hopf-integ-2}
  S(a) q^1 \Lambda_{(1)} \otimes q^2 \Lambda_{(2)}
  & = q^1 \Lambda_{(1)} \otimes a q^2 \Lambda_{(2)}
\end{align}
for all $a \in A$. From these identities, we see that both $S(\Lambda_{(1)}p^1) \otimes \alpha \Lambda_{(2)} p^2$ and $q^1 \Lambda_{(1)} \beta \otimes S(q^2 \Lambda_{(2)})$ are separability idempotents. Applying Theorem~\ref{thm:piv-FS-ind-ch} to $(H, S \tau, g)$, we have
\begin{equation*}
  \nu^\tau(V)
  = \chi_V \Big( S \tau S(\Lambda_{(1)}p^1) g \alpha \Lambda_{(2)} p^2 \Big)
  = \chi_V \Big( S \tau (q^1 \Lambda_{(1)} \beta) g S(q^2 \Lambda_{(2)}) \Big)
\end{equation*}
for all $V \in \fdMod(H)$. Hence, by using the former expression, we compute
\begin{align*}
  \nu^\tau(V)
  & = \chi_V(S^2(\tau(\Lambda_{(1)}p^1)) g \cdot \alpha \Lambda_{(2)} p^2)
  = \chi_V(g \cdot \tau(\Lambda_{(1)}p^1) \cdot \alpha \Lambda_{(2)} p^2) \\
  & = \chi_V(g \cdot \tau(\Lambda_{(1)}p^1 \tau(\alpha)) \cdot \Lambda_{(2)} p^2)
  \mathop{=}^{\eqref{eq:q-Hopf-integ-1}} \chi_V(g \cdot \tau(\Lambda_{(1)}p^1) \Lambda_{(2)} p^2 \cdot S\tau(\alpha))) \\
  & = \chi_V(S\tau(\alpha) g \cdot \tau(\Lambda_{(1)}p^1) \Lambda_{(2)} p^2)
\end{align*}
Note that the formula of Mason and Ng in \cite{MR2104908} does not involve $g$. To exclude $g$ from the above formula of $\nu^\tau(V)$, we require:

\begin{lemma}
  \label{lem:can-piv-2}
  Fix $p \in \{ p_L, p_R \}$ and $q \in \{ q_L, q_R \}$. Then we have
  \begin{equation}
    \label{eq:q-Hopf-integ-3}
    g^{-1} S(\beta) = S(\Lambda_{(1)} p^1) \Lambda_{(2)} p^2, \quad
    S(\alpha)g = S(q^2\Lambda_{(2)}) q^1\Lambda_{(1)}
  \end{equation}
\end{lemma}
\begin{proof}
  The first identity is proved in \cite{MR2095575} (where our $g$ appears as $g^{-1}$) and the second can be proved in a similar way. For the sake of completeness, we give a detailed proof of the second identity.

  Let $V$ be a simple $H$-module and set $c = c_V(1)$. The map
  \begin{equation*}
    e_1: V \otimes V^\vee \to k,
    \quad e_2(v \otimes f) = \langle q^2\Lambda_{(2)} f, q^1 \Lambda_{(1)} v \rangle
    \quad (v \in V, f \in V^\vee)
  \end{equation*}
  is an $H$-linear map such that $e_1(c) = \dim_k(V)$. On the other hand, by the definition of the canonical pivotal structure, we see that
  \begin{equation*}
    \begin{CD}
      e_2: V \otimes V^\vee @>{j_V \otimes \id_{V^\vee}}>> V^{\vee\vee} \otimes V^\vee @>{e_V}>> k
    \end{CD}
  \end{equation*}
  has the same property. Since $\Hom_H(V \otimes V^\vee, k) \cong \Hom_H(V, V) \cong k$, we have $e_1 = e_2$. This implies that $\langle f, S(\alpha)g v \rangle = \langle f, S(q^2 \Lambda_{(2)}) t^1 \Lambda_{(1)} v \rangle$ holds for all $f \in V^\vee$ and $v \in V$. In conclusion, $S(\alpha)g = S(q^2\Lambda_{(2)}) q^1\Lambda_{(1)}$ holds on each simple module $V$. Since $H$ is semisimple, the identity holds in $H$.
\end{proof}

By Lemmas~\ref{lem:can-piv-1} and~\ref{lem:can-piv-2}, we have $S\tau(\alpha) g = \tau(S(\alpha)g) = S(\tau(q^2\Lambda'_{(2)})) \cdot \tau(q^1 \Lambda'_{(1)})$, where $\Lambda' = \Lambda$ is a copy of $\Lambda$. Hence we compute:
\begin{align*}
  \nu^\tau(V)
  & = \chi_V(S(\tau(q^2\Lambda'_{(2)})) \cdot \tau(q^1 \Lambda'_{(1)})
  \cdot \tau(\Lambda_{(1)}p^1) \Lambda_{(2)} p^2) \\
  & = \chi_V(\tau(q^1 \Lambda'_{(1)} \Lambda_{(1)}p^1)
  \cdot \Lambda_{(2)} p^2 S(\tau(q^2\Lambda'_{(2)}))) \\
  & = \chi_V(\tau(q^1 \Lambda'_{(1)} \Lambda_{(1)}p^1 \tau(q^2\Lambda'_{(2)})) \cdot \Lambda_{(2)} p^2) \\
  & = \chi_V(\tau(q^1 \Lambda'_{(1)} \Lambda_{(1)}p^1)
  q^2\Lambda'_{(2)} \Lambda_{(2)} p^2)
\end{align*}
Since $\Delta: H \to H \otimes H$ is an algebra map, we have $\Lambda_{(1)}' \Lambda_{(1)} \otimes \Lambda_{(2)}' \Lambda_{(2)} = \Delta(\Lambda' \Lambda) = \varepsilon(\Lambda') \Delta(\Lambda) = \Delta(\Lambda)$. Hence, we finally obtain $\nu^\tau(V) = \chi_V (\tau(q^1 \Lambda_{(1)}p^1) q^2 \Lambda_{(2)} p^2)$. Letting $\tau = \id_H$, we recover the results of Mason and Ng \cite{MR2104908}. Assuming $H$ to be a Hopf algebra, we recover the results of Sage and Vega \cite{MR2879228}.

\section{Coalgebras}
\label{sec:coalgebras}

\subsection{Copivotal Coalgebras}

In this section, we introduce the dual notion of pivotal algebras and study the Frobenius--Schur theory for them. For the reader's convenience, we briefly recall some basic results on coalgebras.

Given a coalgebra $C$, we denote by $\Com(C)$ the category of {\em right} $C$-comodules and by $\fdCom(C)$ its full subcategory of finite-dimensional objects. We express the coaction of $V \in \Com(C)$ as
\begin{equation*}
  \rho_V: V \to V \otimes C,
  \quad v \mapsto v_{(0)} \otimes v_{(1)}
  \quad (v \in V)
\end{equation*}

The {\em convolution product} of $\lambda, \mu \in C^\vee$ is defined by $\langle \lambda \star \mu, c \rangle = \langle \lambda, c_{(1)} \rangle \langle \mu, c_{(2)} \rangle$ for all $c \in C$. $C^\vee$ is an algebra, called the {\em dual algebra}, with multiplication $\star$ and unit $\varepsilon$. The algebra $C^\vee$ acts from the left on each $V \in \Com(C)$ by
$\text{$\rightharpoonup$}: C^\vee \otimes V \to V$, $\lambda \rightharpoonup v = v_{(0)} \langle \lambda, v_{(1)} \rangle$ ($\lambda \in C^\vee, v \in V$). This defines a $k$-linear fully faithful functors $\Com(C) \to \Mod(C^\vee)$ and
\begin{equation}
  \label{eq:com-to-mod}
  \fdCom(C) \to \fdMod(C^\vee)
\end{equation}
which are not equivalences in general. If $C$ is finite-dimensional, then these functors are isomorphisms of categories. See,  e.g., \cite{MR1786197} for details.

Fix a basis $\{ v_i \}_{i = 1, \dotsc, n}$ of $V \in \fdCom(C)$. Then we can define $c_{i j} \in C$ by $\rho_V(v_j) = \sum_{i = 1}^n v_i \otimes c_{i j}$ ($j = 1, \dotsc, n$). The matrix $(c_{i j})$ is called the {\em matrix corepresentation} of $V$ with respect to the basis $\{ v_i \}$. By the definition of comodules, we have
\begin{equation}
  \label{eq:matrix-corep}
  \Delta(c_{i j}) = \sum_{s = 1}^n c_{i s} \otimes c_{s j},
  \quad \varepsilon(c_{i j}) = \delta_{i j}
\end{equation}
for all $i, j = 1, \dotsc, n$. Hence $C_V = \mathrm{span}_k \{ c_{i j} \mid i, j = 1, \dotsc, n \}$ is a subcoalgebra of $C$. We call $C_V$ the {\em coefficient subcoalgebra} of $C$. $C_V$ has a special element $t_V = \sum_i c_{i i}$, called the {\em character} of $V$. If we regard $V$ as a left $C^\vee$-module via \eqref{eq:com-to-mod} and denote its character by $\chi_V$, then we have
\begin{equation}
  \label{eq:coalg-character}
  \chi_V(\lambda)
  = \langle \lambda, c_{11} \rangle + \dotsb + \langle \lambda, c_{n n} \rangle
  = \lambda(t_V)
\end{equation}
for all $\lambda \in C^\vee$.

Now let $V$ be a finite-dimensional vector space with basis $\{ v_i \}_{i = 1, \dotsc, n}$ and let $\{ v^i \}$ denote the dual basis. Then $\coEnd(V) = V^\vee \otimes V$ has a basis $e_{i j} = v^i \otimes v_j$ ($i, j = 1, \dotsc, n$) and turns into a coalgebra with $\Delta(e_{i j}) = \sum_{s = 1}^n e_{i s} \otimes e_{s j}$, $\varepsilon(e_{i j}) = \delta_{i j}$. $\coEnd(V)$ coacts on $V$ from the right by
\begin{equation*}
  V \to V \otimes \coEnd(V),
  \quad v_j \mapsto \sum_{j = 1}^n v_i \otimes e_{i j}
  \quad (j = 1, \dotsc, n)
\end{equation*}
Suppose that $C$ coacts on $V$. Let $(c_{i j})$ be the matrix corepresentation of $V$ with respect to $\{ v_i \}$. By \eqref{eq:matrix-corep}, the linear map $\phi: \coEnd(V) \to C$, $\phi(e_{i j}) = c_{i j}$ is a coalgebra map. Conversely, if a coalgebra map $\phi: \coEnd(V) \to C$ is given, $V$ is a right $C$-comodule by
\begin{equation*}
  \rho: V \to V \otimes C,
  \quad v_j \mapsto \sum_{i = 1}^n v_i \otimes \phi(e_{i j})
  \quad (j = 1, \dotsc, n)
\end{equation*}
These constructions give a bijection between the set of linear maps $\rho: V \to V \otimes C$ making $V$ into a right $C$-comodule and the set of coalgebra maps $\phi: \coEnd(V) \to C$.

Suppose that $V \in \fdCom(C)$ is absolutely simple. As the dual of~\eqref{eq:abs-simp-surj}, we have that the corresponding coalgebra map $\phi: \coEnd(V) \to C$ is injective. Let $(c_{i j})$ be the matrix corepresentation of $V$ with respect to some basis of $V$. The injectivity of $\phi$ implies that the set $\{ c_{i j} \}$ is linearly independent.

Now we introduce {\em copivotal coalgebras} as the dual notion of pivotal algebras:

\begin{definition}
  A {\em copivotal coalgebra} is a triple $(C, S, \gamma)$ consisting of a coalgebra $C$, an anti-coalgebra map $S: C \to C$ and a linear map $\gamma: C \to k$ satisfying
  \begin{equation*}
    S^2(c) = \langle \gamma, c_{(1)} \rangle c_{(2)} \langle \overline{\gamma}, c_{(3)} \rangle
    \text{\quad and \quad} \overline{\gamma} = \gamma \circ S
  \end{equation*}
  for all $c \in C$, where $\overline{\gamma}: C \to k$ is the inverse of $\gamma$ with respect to $\star$.
\end{definition}

Let $C = (C, S, \gamma)$ is a copivotal coalgebra. If $V \in \fdCom(C)$, then we can make $V^\vee$ into a right $C$-comodule as follows: First fix a basis $\{ v_i \}$ of $V$ and let $(c_{i j})$ be the matrix corepresentation of $V$ with respect to the basis $\{ v_i \}$. Then we define the coaction of $C$ on $V^\vee$ by
\begin{equation}
  \label{eq:dual-comodule}
  \rho_{V^\vee}: V^\vee \to V^\vee \otimes C,
  \quad \rho_{V^\vee}(v^j) = \sum_{i = 1}^n v^i \otimes S(c_{j i})
  \quad (i = 1, \dotsc, n)
\end{equation}
where $\{ v^i \}$ is the dual basis of $\{ v_i \}$. This coaction does not depend on the choice of the basis and has the following characterization, which is rather useful than the above explicit formula:
\begin{equation*}
  \langle f_{(0)}, v \rangle f_{(1)} = \langle f, v_{(0)} \rangle S(v_{(1)})
  \quad (f \in V^\vee, v \in V)
\end{equation*}
For each $V \in \fdCom(C)$, we define $j_V: V \to V^{\vee\vee}$ by
\begin{equation*}
  \langle j_V(v), f \rangle = \langle f, \gamma \rightharpoonup v \rangle
  \quad (= \langle f, v_{(0)} \rangle \langle \gamma, v_{(1)} \rangle)
  \quad (f \in V^\vee, v \in V)
\end{equation*}
In a similar way as Proposition~\ref{sec:piv-alg-cat-dual}, we prove:

\begin{proposition}
  $\fdMod(C)$ is a category with strong duality over $k$.
\end{proposition}

The triple $C^\vee = (C^\vee, S^\vee, \gamma)$ is a pivotal algebra, which we call the {\em dual pivotal algebra} of $C$. Let $V \in \fdMod(C)$. For all $\lambda \in C^\vee$, $f \in V^\vee$ and $v \in V$, we have
\begin{equation*}
  \langle \lambda \rightharpoonup f, v \rangle
  = \langle f_{(0)}, v \rangle \langle \lambda, f_{(1)} \rangle
  = \langle f, v_{(0)} \rangle \langle \lambda, S(v_{(1)}) \rangle
  = \langle f, S^\vee(\lambda) \rightharpoonup v \rangle
\end{equation*}
This implies that the Functor \eqref{eq:com-to-mod} is in fact a strict duality preserving functor. In what follows, we often regard $\fdCom(C)$ as a full subcategory of $\fdMod(C^\vee)$.

\subsection{FS Indicator for Copivotal Coalgebras}

Let $C = (C, S, \gamma)$ be a copivotal coalgebra, and let $V \in \fdCom(V)$. We denote by $\Bil(V)$ the set of all bilinear forms on $V$ and by $\Bil_C(V)$ its subset consisting of those $b \in \Bil(V)$ satisfying
\begin{equation}
  \label{eq:S-adjoint-bilin-coalg}
  b(v_{(0)}, w) v_{(1)} = b(v, w_{(0)}) S(w_{(1)}) \text{\quad for all $v, w \in V$}
\end{equation}
$\Bil_C(V)$ is the image of $\Hom_C(V, V^\vee) \subset \Hom_k(V, V^\vee)$ under the canonical Isomorphism \eqref{eq:canonical-map}. Define $\Sigma_V: \Bil_C(V) \to \Bil_C(V)$ in the same way as before. Then, for all $b \in \Bil_C(V), v, w \in V$, we have
\begin{equation*}
  \Sigma_V(b)(v, w)
  = b(w, \gamma \rightharpoonup v)
\end{equation*}
Now let $\Bil_{C}^\pm(V)$ be the eigenspace of $\Sigma_V$ with eigenvalue $\pm 1$:
\begin{equation*}
  \Bil_{C}^{\pm}(V) = \{ b \in \Bil_{C}(V) \mid
  \text{$b(w, \gamma \rightharpoonup v) = \pm b(v, w)$ for all $v, w \in V$} \}
\end{equation*}
Then we have $\nu(V) = \dim_k \Bil_C^+(V) - \dim_k \Bil_C^-(V)$ as a counterpart of Proposition~\ref{prop:FS-ind-basic} (b). Now we immediately obtain the following coalgebraic version of Theorem \ref{thm:piv-FS}:

\begin{theorem}
  \label{thm:copiv-FS}
  If $V \in \fdCom(C)$ is absolutely simple, then we have $\nu(V) \in \{ 0, \pm 1 \}$. Moreover, the following are equivalent: \\
  \indent {\rm (1)} $\nu(V) \ne 0$. \\
  \indent {\rm (2)} $V$ is isomorphic to $V^\vee$ as a right $C$-comodule. \\
  \indent {\rm (3)} There exists a non-degenerate bilinear form $b$ on $V$ satisfying~\eqref{eq:S-adjoint-bilin-coalg}. \\
  If one of the above statements holds, then such a bilinear form $b$ is unique up to scalar multiples and satisfies $b(w, \gamma \rightharpoonup v) = \nu(V) \cdot b(v, w)$ for all $v, w \in V$.
\end{theorem}

We prove several statements concerning the FS indicator of $V \in \fdCom(C)$. The proof will be done by reducing to the case of pivotal algebras in the following way: First fix a subcoalgebra $D \subset C$ satisfying
\begin{equation}
  \label{eq:copiv-subcoalg}
  \dim_k(D) < \infty, \quad C_V \subset D \text{\quad and \quad} S(D) \subset D
\end{equation}
Note that such a subcoalgebra $D$ always exists. Indeed, by~\eqref{eq:dual-comodule}, we have $C_{X^\vee} = S(C_X)$ for all $X \in \fdCom(C)$. If $V$ is a subcomodule of $X$, then $C_V$ is a subcoalgebra of $C_X$. Therefore, since $X = V \oplus V^\vee$ is self-dual and has $V$ as a subcomodule, $D = C_{V \oplus V^\vee}$ satisfies \eqref{eq:copiv-subcoalg}.

It is obvious that the triple $D = (D, S|_D, \gamma|_D)$ is a copivotal coalgebra and hence $D^\vee$ is a pivotal algebra. As we remarked in the above, the functor
\begin{equation*}
  \begin{CD}
    F_D: \fdMod(D^\vee) @>{\cong}>{\eqref{eq:com-to-mod}}> \fdCom(D) @>{\text{inclusion}}>> \fdCom(C)
  \end{CD}
\end{equation*}
is a $k$-linear fully faithful strict duality preserving functor. Hence, by Proposition~\ref{prop:FS-ind-inv}, $F_D$ preserves the FS indicator. By {\em regarding $V$ as a left $D^\vee$-module}, we mean taking $V_0 \in \fdMod(D^\vee)$ such that $F_D(V_0) = V$ and then identifying $V_0$ with $V$.

Now we prove an analogue of Theorem~\ref{thm:piv-Tr-S}. Let $\reg_C$ denote the coalgebra $C$ regarded as a right $C$-comodule by the comultiplication.

\begin{theorem}
  \label{thm:copiv-Tr-S}
  Let $C = (C, S, \gamma)$ be a finite-dimensional copivotal coalgebra. \\
  {\rm (a)} $\nu(\reg_C^\vee) = \Trace(Q)$, where $Q: C \to C$, $c \mapsto S(c_{(1)}) \langle \gamma, c_{(2)} \rangle$. \\
  {\rm (b)} If $C$ is co-Frobenius, then $\reg_C \cong \reg_C^\vee$ as right $C$-comodules. Hence, $\nu(\reg_C) = \nu(\reg_C^\vee)$.
\end{theorem}
\begin{proof}
  (a) Write $A = C^\vee$ and regard the right $C$-comodule $\reg_C^\vee$ as a left $A$-module via~\eqref{eq:com-to-mod}. To avoid confusion, we denote by $\rightharpoonup$ the action of $A$ on $\reg_A$ and by $\rightharpoondown$ that on $\reg_C^\vee$. Since the coaction of $\mu \in \reg_C^\vee$ is characterized as $\langle \mu_{(0)}, c \rangle \mu_{(1)} = \langle \mu, c_{(1)} \rangle S(c_{(2)})$, the action $\rightharpoondown: A \times \reg_C^\vee \to \reg_C^\vee$ is given by
  \begin{equation*}
    \langle f \rightharpoondown \mu, c \rangle
    = \langle f, S(c_{(2)}) \rangle \langle \mu, c_{(1)} \rangle
    \quad (f \in A, \mu \in \reg_C^\vee, c \in C)
  \end{equation*}
  Consider the map $S^\vee: \reg_A \to \reg_C^\vee$. For $\lambda \in A$, $f \in \reg_A$ and $c \in C$, we have
  \begin{gather*}
    \langle S^\vee(\lambda \rightharpoonup f), c \rangle
    = \langle \lambda \star f, S(c) \rangle
    = \langle \lambda, S(c_{(2)}) \rangle \langle f, S(c_{(1)}) \rangle \\
    = \langle \lambda, S(c_{(2)}) \rangle \langle S^\vee(f), c_{(1)} \rangle
    = \langle \lambda \rightharpoondown S^\vee(f), c \rangle
  \end{gather*}
  and therefore $S^\vee: \reg_A \to \reg_C^\vee$ is an isomorphism of $A$-modules. Applying Theorem~\ref{thm:piv-Tr-S} to $A = C^\vee$, we see that $\nu(\reg_C^\vee) = \nu(\reg_A)$ is equal to the trace of the map $Q': A \to A$, $\lambda \mapsto S^\vee(\lambda) \star \gamma$ ($\lambda \in C^\vee$). Since $\langle Q'(\lambda), c \rangle = \langle \lambda, S(c_{(1)}) \rangle \langle \gamma, c_{(2)} \rangle = \langle \lambda, Q(c) \rangle$, $Q'$ is the dual map of $Q$. Hence, we have $\nu(\reg_C^\vee) = \Trace(Q') = \Trace(Q)$.

  (b) If $C$ is co-Frobenius, then $A$ is Frobenius. Thus, by Theorem~\ref{thm:piv-Tr-S}, there is an isomorphism $\varphi: \reg_A \to \reg_A^\vee$ of $A$-modules. In the proof of (1), we see that $S^\vee: \reg_A \to \reg_C^\vee$ is an isomorphism of $A$-modules. Regarding them as isomorphisms in the category $\fdCom(C)$, we obtain an isomorphism
  \begin{equation*}
    \begin{CD}
      \reg_C @>{j}>> \reg_C^{\vee\vee} @>{S^{\vee\vee}}>> \reg_A^\vee @>{\varphi}>> \reg_A @>{S^\vee}>> \reg_C^\vee
    \end{CD}
  \end{equation*}
  of right $C$-comodules.
\end{proof}

The following is a coalgebraic version of Theorem~\ref{thm:piv-Tr-Sv}.

\begin{theorem}
  \label{thm:copiv-Tr-Sv}
  Let $V \in \fdCom(C)$ be an absolutely simple comodule and suppose that $V$ is self-dual. Then, by \eqref{eq:dual-comodule}, the map
  \begin{equation*}
    S_V: C_V \to C_V, \quad S_V(c) = S(c) \quad (c \in C)
  \end{equation*}
  is well-defined. We also define
  \begin{equation*}
    Q_V: C_V \to C_V, \quad Q_V(c) = S(\gamma \rightharpoonup c)
    \quad (= S(c_{(1)}) \gamma(c_{(2)})) \quad (c \in C)
  \end{equation*}
  Then we have:
  \begin{equation*}
    \mathrm{(1)} \Trace(S_V)  = \nu(V) \cdot \gamma(t_V) \qquad
    \mathrm{(2)} \Trace(Q_V) = \nu(V) \cdot \dim_k(V)
  \end{equation*}
\end{theorem}
\begin{proof}
  (1) We regard $V$ as a left $(C_V)^\vee$-module and denote its character by $\chi_V$. Applying Theorem~\ref{thm:piv-Tr-Sv} to the dual pivotal algebra $(C_V)^\vee$ and by using \eqref{eq:coalg-character}, we have $\Trace(S_V) = \nu(V) \cdot \chi_V(\gamma) = \nu(V) \cdot \gamma(t_V)$.

  (2) We regard $C_V$ as a right $C_V$-comodule. Since $C_V \cong \coEnd(V)$ is co-Frobenius, by Theorem~\ref{thm:copiv-Tr-S}, we have $\nu(C_V) = \Trace(Q_V)$. Let $d = \dim_k(V)$. Since $C_V \cong V^{\oplus d}$ as a right $C$-comodule, we have $\nu(C_V) = \nu(V) d$. Hence, $\Trace(Q_V) = \nu(V) d$.
\end{proof}

Applying Corollary~\ref{cor:Tr-S} to the dual pivotal algebra of $C$, we have:

\begin{corollary}
  \label{cor:Tr-S-copiv}
  Suppose that $k$ is algebraically closed and that $C = (C, S, \gamma)$ is a finite-dimensional cosemisimple copivotal coalgebra. Let $\{ V_i \}_{i = 1, \dotsc, n}$ be a complete set of representatives of the isomorphism classes of simple right $C$-comodules. Then
  \begin{equation*}
    \Trace(S) = \sum_{i = 1}^n \nu(V_i) \gamma(t_i)
  \end{equation*}
  where $t_i = t_{V_i}$ is the character of $V_i$.
\end{corollary}

\subsection{Coseparable Copivotal Coalgebras}

A coalgebra $C$ is said to be {\em coseparable} if it has a {\em coseparability idempotent}, {\em i.e.}, a bilinear form $\lambda: C \times C \to k$ satisfying $c_{(1)} \lambda(c_{(2)}, d) = \lambda(c, d_{(1)}) d_{(2)}$ and $\lambda(c_{(1)}, c_{(2)}) = \varepsilon(c)$ for all $c, d \in C$. If such a form exists, then the forgetful functor $\Com(C) \to \Vect(k)$ is separable with section $\Pi_{V,W}: \Hom_k(V, W) \to \Hom_C(V, W)$ given by
\begin{equation*}
  \begin{CD}
    \Pi_{V,W}(f):
    V @>{\rho_V}>> V \otimes C
    @>{f \otimes \id_V}>> W \otimes C
    @>{\rho_W}>> W \otimes C \otimes C
    @>{\id_W \otimes \lambda}>> W
  \end{CD}
\end{equation*}
for $f \in \Hom_k(V, W)$. The following theorem can be proved by the arguments of \S\ref{subsec:sep-functor}. Nevertheless, to avoid notational difficulties, we do not use $\Pi$ and prove the theorem by reducing to Theorem \ref{thm:piv-FS-ind-ch}.

\begin{theorem}
  \label{thm:copiv-FS-ind-ch}
  If $C = (C, S, \gamma)$ is a coseparable copivotal coalgebra with coseparability idempotent $\lambda$, then, for all $V \in \fdCom(C)$, we have
  \begin{equation*}
    \nu(V)
    = \lambda(S(\gamma \rightharpoonup t_{V(1)}), t_{V(2)})
    \quad \Big( = \lambda(S(t_{V(1)}), t_{V(3)}) \gamma(t_{V(2)}) \Big)
  \end{equation*}
\end{theorem}
\begin{proof}
  Fix a subcoalgebra $D$ of $C$ satisfying~\eqref{eq:copiv-subcoalg}. $D$ is coseparable with $\lambda_D = \lambda|_{D \times D}$. Since $D$ is finite-dimensional, there exist finite number of linear maps $\lambda_i', \lambda_i'': D \to k$ such that
  \begin{equation*}
    \lambda_D(x, y)
    = \sum_{i} \lambda_i'(x) \lambda_i''(y)
  \end{equation*}
  for all $x, y \in D$. It is easy to see that $E = \sum_{i} \lambda_i' \otimes \lambda_i''$ is a separability idempotent for the dual pivotal algebra $D^\vee$. Now we regard $V$ as a left $D^\vee$-module and denote its character by $\chi_V$. Applying Theorem \ref{thm:piv-FS-ind-ch} to $D$, we obtain
  \begin{equation*}
    \nu(V) = \sum_i \chi_V (S_X^\vee(\lambda_i') \star \gamma \star \lambda_i'')
  \end{equation*}
  Now the desired formula is obtained by using~\eqref{eq:coalg-character}.
\end{proof}

A {\em copivotal Hopf algebra} is a Hopf algebra $H = (H, \Delta, \varepsilon, S)$ equipped with an algebra map $\gamma: H \to k$ satisfying $S^2(x) = \langle \gamma, x_{(1)} \rangle x_{(2)} \langle \gamma, S(x_{(3)}) \rangle$ for all $x \in H$. Since $\gamma$ is an algebra map, $\gamma \circ S$ is the inverse of $\gamma$ with respect to the convolution product. Therefore a copivotal Hopf algebra is a copivotal coalgebra.

A {\em Haar functional} of a Hopf algebra $H$ is a linear map $\lambda: H \to k$ satisfying
$\langle \lambda, 1 \rangle = 1$ and
$\langle \lambda, x_{(1)} \rangle x_{(2)} = \varepsilon(x) 1 = x_{(1)} \langle \lambda, x_{(2)} \rangle$
for all $x \in H$. If $\lambda$ is a Haar functional of $H$, then the map
\begin{equation*}
  \tilde{\lambda}: H \times H \to k,
  \quad \tilde{\lambda}(x, y) = \langle \lambda, S(x) y \rangle
  \quad (x, y \in H)
\end{equation*}
is a coseparability idempotent of the coalgebra $H$. Note that we have
\begin{equation*}
  S^2(x_{(1)}) \langle \gamma, x_{(2)} \rangle
  = \langle \gamma, x_{(1)} \rangle x_{(2)} \langle \gamma, S(x_{(3)}) \rangle \langle \gamma, x_{(4)} \rangle
  = \langle \gamma, x_{(1)} \rangle x_{(2)}
\end{equation*}
for all $x \in H$. The following corollary is a direct consequence of Theorem~\ref{thm:copiv-FS-ind-ch}:

\begin{corollary}
  \label{cor:FS-ind-Haar-int}
  Regard a copivotal Hopf algebra $H = (H, \Delta, \varepsilon, S; \gamma)$ as a copivotal coalgebra. If there exists a Haar functional $\lambda: H \to k$ on $H$, then we have
  \begin{equation*}
    \nu(V) = \langle \gamma, t_{V(1)} \rangle \langle \lambda, t_{V(2)} t_{V(3)} \rangle
  \end{equation*}
  for all $V \in \fdCom(H)$.
\end{corollary}

We shall explain how can we obtain~\eqref{eq:FS-formula-cpt} from Corollary~\ref{cor:FS-ind-Haar-int}.

\begin{example}
  We work over $\mathbb{C}$. Let $G$ be a compact group. A function $f: G \to \mathbb{C}$ is said to be {\em representative} if there exist finite number of functions $f_i, g_i: G \to \mathbb{C}$ such that $f(x y) = \sum_i f_i(x) g_i(y)$ for all $x, y \in G$. We denote by $R(G)$ the algebra of continuous representative functions on $G$. $R(G)$ is in fact a Hopf algebra; the comultiplication, the counit and the antipode are given by
  \begin{equation*}
    f_{(1)}(x) f_{(2)}(y)= f(x y),
    \quad \varepsilon(f) = f(1),
    \quad S(f)(x) = f(x^{-1})
  \end{equation*}
  for $f \in R(G)$, $x, y \in G$. Define $\lambda: R(G) \to \mathbb{C}$ by $\lambda(f) = \int_G f(x) d \mu(x)$, where $\mu$ is the normalized Haar measure on $G$. We see that $\lambda$ is a Haar functional of $R(G)$ (in fact, this is the origin of this term).

  The group $G$ acts continuously from the left on each $V \in \fdCom(R(G))$ by $x \cdot v = v_{(1)}(x) \cdot v_{(0)}$ ($x \in G, v \in V$). Conversely, if $V$ is a finite-dimensional continuous representation of $G$, then $R(G)$ coacts from the right on $V$. If we fix a basis $\{ v_i \}_{i = 1, \dotsc, n}$ of $V$, the coaction of $R(G)$ is described as follows: Define $f_{i j}: G \to \mathbb{C}$ by
  \begin{equation}
    \label{eq:FS-ind-cpt-grp-1}
    x \cdot v_i = \sum_{i = 1}^n f_{i j}(x) v_j \quad (x \in G, i = 1, \dotsc, n)
  \end{equation}
  Then each $f_{i j}$ is an element of $R(G)$. The coaction of $R(G)$ on $V$ is defined by
  \begin{equation*}
    V \to V \otimes R(G), \quad v_i \mapsto \sum_{j = 1}^n v_j \otimes f_{i j} \quad (i = 1, \dotsc, n)
  \end{equation*}
  These correspondences give an isomorphism of categories with duality over $\mathbb{C}$ between $\fdCom(R(G))$ and the category of continuous representations of $G$.

  Now let $V$ be a continuous representation of $G$ with character $\chi_V$. Regarding $V$ as a right $R(G)$-comodule via the above category isomorphism, we obtain $\nu(V) = \lambda(t_{V(1)} t_{V(2)})$ by Corollary~\ref{cor:FS-ind-Haar-int}. To compute this value, we fix a basis $\{ v_i \}_{i = 1, \dotsc, n}$ of $V$ and define $f_{i j}$ by \eqref{eq:FS-ind-cpt-grp-1}. Then $t_V = f_{1 1} + \dotsb + f_{n n}$. Hence, by~\eqref{eq:matrix-corep}, we compute
  \begin{align}
    \label{eq:FS-ind-cpt-grp-2}
    \nu(V)
    = \lambda(t_{V(1)} t_{V(2)})
    = \int_G^{} \left( \sum_{i, j = 1}^n f_{i j}(x)f_{ji}(x) \right) d \mu(x)
  \end{align}
  Since the action of $x \in G$ is represented by $\rho(x) = (f_{i j}(x))_{i, j = 1, \dotsc, n}$, we have
  \begin{equation*}
    \chi_V(x^2)
    = \Trace \Big( \rho(x)^2 \Big)
    = \sum_{i, j = 1}^n f_{i j}(x)f_{ji}(x)
  \end{equation*}
  Substituting this to~\eqref{eq:FS-ind-cpt-grp-2}, we obtain~\eqref{eq:FS-formula-cpt}.
\end{example}

\section{Quantum $SL_2$}
\label{sec:quantum-sl_2}

\subsection{The Hopf Algebra $\mathcal{O}_q(SL_2)$}

In this section, we give some applications of our results to the quantum coordinate algebra $\mathcal{O}_q(SL_2)$ and the quantized universal enveloping algebra $U_q(\mathfrak{sl}_2)$. For details on these Hopf algebras, we refer the reader to \cite{MR1321145} and \cite{MR1492989}.

Throughout, the base field $k$ is assumed to be an algebraically closed field of characteristic zero. $q \in k$ denotes a fixed non-zero parameter that is not a root of unity. We use the following standard notations:
\begin{equation*}
  [n]_q = \frac{q^n - q^{-n}}{q - q^{-1}},
  \quad [n]_q! = [n]_q \cdot [n - 1]_q!
  \quad (n \ge 1),
  \quad [0]_q! = 1
\end{equation*}
for $n \in \mathbb{N}_0 = \{ 0, 1, 2, \dotsc \}$.

The quantum coordinate algebra $\mathcal{O}_q(SL_2)$ is a Hopf algebra defined as follows: As an algebra, it is generated by $a$, $b$, $c$ and $d$ with relations
\begin{gather*}
  a b = q b a, \quad a c = q c a, \quad b d = q d b, \quad c d = q d c, \quad b c = c b \\
  a d - q b c = 1 = d a - q^{-1} b c
\end{gather*}
The comultiplication $\Delta$ and the counit $\varepsilon$ are defined by
\begin{gather*}
  \Delta(a) = a \otimes a + b \otimes c, \quad \Delta(b) = a \otimes b + b \otimes d,
  \quad \varepsilon(a) = 1, \quad \varepsilon(b) = 0 \\
  \Delta(c) = c \otimes a + d \otimes c, \quad \Delta(d) = c \otimes b + d \otimes d,
  \quad \varepsilon(c) = 0, \quad \varepsilon(d) = 1
\end{gather*}
and the antipode $S$ is given by
\begin{equation*}
  S(a) = d, \quad S(b) = -q^{-1} b, \quad S(c) = - q c, \quad S(d) = a
\end{equation*}
We define an algebra map $\gamma: \mathcal{O}_q(SL_2) \to k$ by
\begin{equation*}
  \gamma(a) = q^{-1}, \quad \gamma(b) = \gamma(c) = 0, \quad \gamma(d) = q
\end{equation*}
One can check that $\mathcal{O}_q(SL_2)$ is a copivotal Hopf algebra with $\gamma$. In what follows, we determine the FS indicator of simple $\mathcal{O}_q(SL_2)$-comodules.

For each $\ell \in \frac{1}{2}\mathbb{N}_0$, we put $I_\ell = \{ - \ell, - \ell + 1, \dotsc, \ell - 1, \ell \}$ and
\begin{equation*}
  X_\ell = \mathrm{span}_k \{ a^{\ell - i} b^{\ell + i} \mid i \in I_\ell \} \subset \mathcal{O}_q(SL_2)
\end{equation*}
$X_\ell$ is a right coideal and hence it is a right $\mathcal{O}_q(SL_2)$-comodule. It is known that each $X_\ell$ is simple and $\{ X_\ell \mid \ell \in \frac{1}{2}\mathbb{N}_0 \}$ is a complete set of representatives of the isomorphism classes of simple right $\mathcal{O}_q(SL_2)$-comodules. This implies, in particular, that $X_\ell$ is self-dual.

In this section, we first prove the following result:

\begin{theorem}
  \label{thm:FS-ind-Oq}
  $\nu(X_\ell) = (-1)^{2 \ell}$.
\end{theorem}

By Theorem~\ref{thm:copiv-FS}, this result reads as follows: For each $\ell \in \frac{1}{2}\mathbb{N}_0$, there exists a non-degenerate bilinear form $\beta$ on $X_\ell$ satisfying $\beta(x_{(0)}, y) x_{(1)} = \beta(x, y_{(0)}) S(y_{(1)})$ and $b(y, \gamma \rightharpoonup x) = (-1)^{2\ell} \cdot b(x, y)$ for all $x, y \in X_\ell$.

To prove Theorem~\ref{thm:FS-ind-Oq}, we need a matrix corepresentation of $X_\ell$. For each $i \in I_\ell$, we fix a square root $\mu_i \in k$ of $[\ell + i]_{q^{-2}}!$ and take
\begin{equation*}
  x_i^{(\ell)} = \left[ \begin{array}{c} 2 \ell \\ \ell + i \end{array} \right]^{1/2}_{q^{-2}}
  a^{\ell - i} b^{\ell + i} \quad (i \in I_\ell), \text{\quad where \quad}
  \left[ \begin{array}{c} 2 \ell \\ \ell + i \end{array} \right]^{1/2}_{q^{-2}}
  := \frac{\mu_{\ell}}{\mu_{i} \cdot \mu_{-i}}
\end{equation*}
as a basis of $X_\ell$. Define $c_{i j}^{(\ell)}$ by $\Delta(x_j^{(\ell)}) = \sum x_j^{(\ell)} \otimes c_{i j}^{(\ell)}$. The matrix $(c_{i j}^{(\ell)})$ has been explicitly determined and well-studied in relation to unitary representations of a real form of $\mathcal{O}_q(SL_2)$; see,  e.g., \cite[\S4]{MR1492989}. Following {\em loc. cit.}, we have
\begin{equation*}
  c_{i j}^{(\ell)} = \begin{cases}
    N_{i j \ell}^{+} \cdot a^{-i-j} c^{i-j} \cdot p_{\ell + j} (\zeta; q^{-2(i - j)}, q^{2(i + j)} | q^{-2})
    & (i + j \le 0, i \ge j) \\
    N_{j i \ell}^{+} \cdot a^{-i-j} b^{j-i} \cdot p_{\ell + i} (\zeta; q^{-2(j - i)}, q^{2(i + j)} | q^{-2})
    & (i + j \le 0, i \le j) \\
    N_{j i \ell}^{-} \cdot p_{\ell - i} (\zeta; q^{-2(i - j)}, q^{2(i + j)} | q^{-2}) \cdot c^{i - j} d^{i + j}
    & (i + j \ge 0, i \ge j) \\
    N_{i j \ell}^{-} \cdot p_{\ell - j} (\zeta; q^{-2(i - j)}, q^{2(i + j)} | q^{-2}) \cdot b^{i - j} d^{i + j}
    & (i + j \ge 0, i \le j) \\
  \end{cases}
\end{equation*}
where $\zeta = - q b c$,
\begin{equation*}
  N_{i j \ell}^+ = \frac{q^{-(\ell + j)(j - i)}}{[i - j]_{q^{-2}}!}
  \cdot \frac{\mu_{+i} \, \mu_{- j}}{\mu_{+j} \, \mu_{-i}}, \quad
  N_{i j \ell}^- = \frac{q^{(\ell - j)(j - i)}}{[j - i]_{q^{-2}}!}
  \cdot \frac{\mu_{-i} \, \mu_{+j}}{\mu_{-j} \, \mu_{+i}}
  \quad (= N_{-i,-j, \ell}^+)
\end{equation*}
and $p_m$ is the {\em little $q$-Jacobi polynomial} \cite[\S2]{MR1492989}. We omit the definition of $p_m$; in what follows, we need only the fact that $p_m(\zeta; q_1, q_2 | q_3)$ ($q_i \in k$) is a polynomial of $\zeta$. Note that, since $S(\zeta) = \zeta$, we have
\begin{equation}
  \label{eq:q-Jacobi-1}
  S(p_m(\zeta; q_1, q_2 | q_3)) = p_m(S(\zeta); q_1, q_2 | q_3) = p_m(\zeta; q_1, q_2 | q_3)
\end{equation}
Since $f \mapsto (\gamma \rightharpoonup f)$ ($f \in \mathcal{O}_q(SL_2)$) is an algebra map, we also have
\begin{equation}
  \label{eq:q-Jacobi-2}
  \gamma \rightharpoonup p_m(\zeta; q_1, q_2 | q_3) = p_m(\gamma \rightharpoonup \zeta; q_1, q_2 | q_3) = p_m(\zeta; q_1, q_2 | q_3)
\end{equation}

\begin{proof}[Proof of Theorem~\ref{thm:FS-ind-Oq}]
  Let $C_\ell$ be the coefficient subcoalgebra of $X_\ell$. Define
  \begin{equation*}
    Q_\ell: C_\ell \to C_\ell
    \quad Q_\ell(f) = S(\gamma \rightharpoonup f)
    \quad (f \in C_\ell)
  \end{equation*}
  $Q_\ell$ is well-defined since $X_\ell$ is self-dual. By Theorem~\ref{thm:copiv-Tr-Sv}, we have
  \begin{equation*}
    \nu(X_\ell) = \frac{\Trace(Q_\ell)}{\dim_k(X_\ell)} = \frac{\Trace(Q_\ell)}{2 \ell + 1}
  \end{equation*}
  By~\eqref{eq:q-Jacobi-1}, \eqref{eq:q-Jacobi-2} and the above description of $c_{i j}^{(\ell)}$, we have
  \begin{equation*}
    Q_\ell(c_{i j}^{(\ell)})
    = S(\gamma \rightharpoonup c_{i j}^{(\ell)})
    = \text{(constant)} \times S(c_{i j}^{(\ell)})
    = \text{(constant)} \times c_{-j, -i}^{(\ell)}
  \end{equation*}
  for all $i, j \in I_\ell$. Recall that $\{ c_{i j}^{(\ell)} \}$ is a basis of $C_\ell$ since $X_\ell$ is simple. The above computation means that $Q_\ell$ is represented by a generalized permutation matrix with respect to this basis.

  Note that $(i, j) = (-j, -i)$ if and only if $j = -i$. If $i \ge 0$, then
  \begin{align*}
    Q_\ell(c_{i, -i}^{(\ell)})
    & = q^{-2i} \cdot S \Big( N_{i,-i,\ell}^{+} \cdot c^{2i} \cdot p_{\ell - i} (\zeta; q^{-4i}, 1 | q^{-2}) \Big) \\
    & = q^{-2i} \cdot N_{i,-i,\ell}^{+} \cdot p_{\ell - i} (\zeta; q^{-4i}, 1 | q^{-2}) \cdot (-q)^{2 i}
    = (-1)^{2i} \cdot c_{i, -i}^{(\ell)}
  \end{align*}
  Since $\ell - i \in \mathbb{Z}$, we have $(-1)^{2 i} = (-1)^{2 \ell}$. In a similar way, we also have $Q_\ell(c_{i,-i}^{(\ell)}) = (-1)^{2 \ell}$ for $i < 0$. Hence we obtain $\Trace(Q_\ell) = (-1)^{2 \ell} \cdot (2 \ell + 1)$.
\end{proof}

$\mathcal{O}_q(SL_2)$ has a Hopf algebra automorphism $\tau$ given by
\begin{equation*}
  \tau(a) = a, \quad \tau(b) = - b, \quad \tau(c) = - c, \quad \tau(d) = d
\end{equation*}
$\tau$ is an involution such that $\gamma \circ \tau = \gamma$ and hence the $\tau$-twisted FS indicator $\nu^\tau(X)$ is defined for each $X \in \fdCom(\mathcal{O}_q(SL_2))$. Replacing $S$ in the proof of Theorem~\ref{thm:FS-ind-Oq} with $S \circ \tau$, we have the following theorem:

\begin{theorem}
  \label{thm:tw-FS-ind-Oq}
  $\nu^\tau(X_\ell) = +1$.
\end{theorem}

By Theorem~\ref{thm:copiv-FS}, this result reads as follows: For each $\ell \in \frac{1}{2}\mathbb{N}_0$, there exists a non-degenerate bilinear form $\beta$ on $X_\ell$ satisfying $\beta(x_{(0)}, y) \tau(x_{(1)}) = \beta(x, y_{(0)}) S(y_{(1)})$ and $\beta(w, \gamma \rightharpoonup v) = \beta(v, w)$ for all $v, w \in X_\ell$.

\begin{remark}
  \label{rem:FS-ind-Oq}
  The character $t_\ell$ of $X_\ell$ is given by
  \begin{equation*}
    t_\ell = \sum_{i \in I_\ell} c_{i i}^{(\ell)}
    = \sum_{i \in I_\ell, i \ge 0} a^{2 i} p_{\ell + i}(\zeta; 1, q^{-4i} | q^{-2})
    + \sum_{i \in I_\ell, i >   0} p_{\ell + i}(\zeta; 1, q^{4i} | q^{-2}) d^{2 i}
  \end{equation*}
  One can prove Theorems~\ref{thm:FS-ind-Oq} and \ref{thm:tw-FS-ind-Oq} by Theorem~\ref{thm:copiv-FS-ind-ch} and its corollary (see \cite[\S4]{MR1492989} for a description of the Haar functional on $\mathcal{O}_q(SL_2)$). However, the computation will become more difficult than the above proof.
\end{remark}

\subsection{The Hopf Algebra $U_q(\mathfrak{sl}_2)$}

The quantized enveloping algebra $U_q(\mathfrak{sl}_2)$ is a Hopf algebra defined as follows: As an algebra, it is generated by $E$, $F$, $K$ and $K^{-1}$ with relations $K K^{-1} = 1 = K^{-1} K$,
\begin{gather*}
  K E K^{-1} = q^2 E,
  \quad K F K^{-1} = q^{-2} F
  \text{\quad and \quad} E F - F E = \frac{K - K^{-1}}{q - q^{-1}}
\end{gather*}
The comultiplication $\Delta$, the counit $\varepsilon$ and the antipode $S$ are given by
\begin{gather*}
  \Delta(K) = K \otimes K,
  \quad \Delta(E) = E \otimes K + 1 \otimes E,
  \quad \Delta(F) = F \otimes 1 + K^{-1} \otimes F \\
  S(K) = K^{-1}, \ S(E) = - E K^{-1}, \ S(F) = - K F,
  \ \varepsilon(K) = 1,
  \ \varepsilon(E) = \varepsilon(F) = 0
\end{gather*}
We have $S^2(u) = K u K^{-1}$ for all $u \in U_q(\mathfrak{sl}_2)$. Hence the Hopf algebra $U_q(\mathfrak{sl}_2)$ is pivotal with pivotal grouplike element $K$.

For each $\ell \in \frac{1}{2} \mathbb{N}_0$, we define a left $U_q(\mathfrak{sl}_2)$-module $V_\ell$ as follows: As a vector space, it has a basis $\{ v_i \}_{i \in I_\ell}$. The action of $U_q(\mathfrak{sl}_2)$ on $V_\ell$ is defined by
\begin{equation*}
  K \cdot v_i = q^{2i} v_i,
  \ E \cdot v_i = [\ell - i + 1]_q v_{i - 1},
  \ F \cdot v_i = [\ell + i + 1]_q v_{i + 1}
  \ (i \in I_\ell)
\end{equation*}
where $v_{\ell+1} = v_{-(\ell+1)} = 0$.

There is a unique Hopf pairing $\langle -, - \rangle: U_q(\mathfrak{sl}_2) \times \mathcal{O}_q(SL_2) \to k$ such that
\begin{gather*}
  \langle K, a \rangle = q^{-1},
  \quad \langle K, d \rangle = q,
  \quad \langle E, c \rangle = 1,
  \quad \langle F, b \rangle = 1 \\
  \langle K, b \rangle = \langle K, c \rangle 
  = \langle E, a \rangle = \langle E, b \rangle = \langle E, d \rangle 
  = \langle F, a \rangle = \langle F, c \rangle = \langle F, d \rangle = 0
\end{gather*}
see \cite[\S4]{MR1492989} and \cite[V.7]{MR1321145}. This pairing induces an algebra map $\varphi: U_q(\mathfrak{sl}_2) \to \mathcal{O}_q(SL_2)^\vee$. Since $\varphi(K) = \gamma$, $\varphi$ is in fact a morphism of pivotal algebras. Hence we obtain a $k$-linear duality preserving functor
\begin{equation*}
  \begin{CD}
    \Phi: \fdCom\Big( \mathcal{O}_q(SL_2) \Big)
    @>>{\eqref{eq:com-to-mod}}>
    \fdMod\Big( \mathcal{O}_q(SL_2)^\vee \Big)
    @>{\varphi^\natural}>>
    \fdMod\Big(U_q(\mathfrak{sl}_2) \Big)
  \end{CD}
\end{equation*}
One has $\Phi(X_\ell) \cong V_\ell$. In particular, $\Phi$ maps simple objects to simple objects. Since $\mathcal{O}_q(SL_2)$ is cosemisimple, the functor $\Phi$ is fully faithful and therefore $\Phi$ preserves the FS indicator. Hence, by Theorem~\ref{thm:FS-ind-Oq}, we have:

\begin{theorem}
  $\nu(V_\ell) = (-1)^{2 \ell}$.
\end{theorem}

By Theorem~\ref{thm:piv-FS-ind-ch}, this result reads as follows: For each $\ell \in \frac{1}{2} \mathbb{N}_0$, there exists a non-degenerate bilinear form $\beta$ on $V_\ell$ satisfying $\beta(u v, w) = \beta(v, S(u) w)$ and $\beta(w, K v) = (-1)^{2 \ell} \cdot \beta(v, w)$ for all $u \in U_q(\mathfrak{sl}_2)$ and $v, w \in V_\ell$.

$U_q(\mathfrak{sl}_2)$ has a Hopf algebra automorphism $\tau$ defined by $\tau(E) = -E$, $\tau(F) = -F$, $\tau(K) = K$. It is obvious that $\tau$ is an involution of the pivotal algebra $U_q(\mathfrak{sl}_2)$. Since $\langle \tau(u), f \rangle = \langle u, \tau(f) \rangle$ for all $u \in U_q(\mathfrak{sl}_2)$ and $f \in \mathcal{O}_q(SL_2)$, $\Phi$ also preserves the $\tau$-twisted FS indicator. Therefore we have:

\begin{theorem}
  $\nu^\tau(V_\ell) = +1$.
\end{theorem}

This result reads as follows: For each $\ell \in \frac{1}{2} \mathbb{N}_0$, there exists a non-degenerate bilinear form $\beta$ on $V_\ell$ satisfying $\beta(\tau(u) v, w) = b(v, S(u)w)$ and $\beta(w, K v) = b(v, w)$ for all $u \in U_q(\mathfrak{sl}_2)$ and $v, w \in V_\ell$.


\section{Conclusions}

As we have briefly reviewed in Section~\ref{sec:introduction}, the celebrated theorem of Frobenius and Schur has several generalizations. To give a category-theoretical understanding of these generalizations, in Section~\ref{sec:categ-with-dual} we have introduced the FS indicator for categories with duality over a field $k$; if $\mathcal{C}$ is a category with duality over $k$, then a linear map $\Trans_{X,Y}: \Hom_\mathcal{C}(X, Y^\vee) \to \Hom_\mathcal{C}(Y, X^\vee)$, $f \mapsto f^\vee \circ j$ is defined for each $X, Y \in \mathcal{C}$. We call $\Trans_{X,Y}$ the transposition map. The FS indicator $\nu(X)$ of $X \in \mathcal{C}$ is defined to be the trace of $\Trans_{X,X}: \Hom_\mathcal{C}(X, X^\vee) \to \Hom_\mathcal{C}(X, X^\vee)$. We have also introduced a general method to twist the given duality by an adjunction, which is a category-theoretical counterpart of several twisted versions of the Frobenius--Schur theorem.

In Section~\ref{sec:piv-alg}, we have introduced the notion of a pivotal algebra. The representation category of a pivotal algebra has duality and therefore the FS indicator is defined for each of its representation. We have given a representation-theoretic interpretation of the FS indicator and a formula of the FS indicator for separable pivotal algebras. These results yield the Frobenius--Schur-type theorems for Hopf algebras, quasi-Hopf algebras, weak Hopf $C^*$-algebras and Doi's group-like algebras. The notion of pivotal algebras is useful to deal with the twisted FS indicator; as a demonstration, we have constructed the twisted Frobenius--Schur theory for quasi-Hopf algebras.

In Section~\ref{sec:coalgebras}, we have introduced the notion of a copivotal coalgebra as the dual notion of a pivotal algebra and gave results for copivotal coalgebras analogous to pivotal algebras. In particular, we have given a representation-theoretic interpretation of the FS indicator and a formula of the FS indicator for coseparable copivotal coalgebras. 

In Section~\ref{sec:quantum-sl_2}, we have applied our results to the quantum coordinate ring $\mathcal{O}_q(SL_2)$ and the quantum enveloping algebra $U_q(\mathfrak{sl}_2)$. For each $\ell \in \frac{1}{2}\mathbb{N}_0$, $\mathcal{O}_q(SL_2)$ has a unique simple right comodule $X_\ell$ of dimension $2 \ell$. We have proved $\nu(X_\ell) = (-1)^{2 \ell}$ and analogous results for the twisted case and $U_q(\mathfrak{sl}_2)$ case. As we have remarked, the Haar functional on $\mathcal{O}_q(SL_2)$ is not used in our proof. We expect that the FS indicator for general $\mathcal{O}_q(G)$ will be determined by using the Haar functional.

\section*{Acknowledgments}

The author would like to thank the referees for careful reading the manuscript. The author is supported by Grant-in-Aid for JSPS Fellows (24$\cdot$3606).


\end{document}